\documentclass[11pt]{amsart}

\usepackage{amsfonts,epsfig}
\usepackage{latexsym}
\usepackage{amssymb}
\usepackage{amsmath}
\usepackage{amsthm}
\usepackage{graphics}
\usepackage[all]{xy}
\usepackage[T2A]{fontenc}
\usepackage{multirow}
\usepackage{float}

\usepackage{array, booktabs, ctable}
\usepackage{hyperref}
\usepackage{color}
\usepackage{mathtools}
\usepackage{enumerate}
\newtheorem{defn}{Definition}[section]

\newtheorem{corollary}[defn]{Corollary}
\newtheorem{lemma}[defn]{Lemma}

\newtheorem{theorem}[defn]{Theorem}

\theoremstyle{definition}

\newcommand{\lmfdbec}[3]{\href{https://www.lmfdb.org/EllipticCurve/Q/#1/#2/#3}{#1.#2#3}}

\newcommand{\Q}{\mathbb Q}
\newcommand{\Z}{\mathbb Z}

\newcommand{\Gal}{\operatorname{Gal}}

\begin{document}

\bibliographystyle{plain}

\title[Torsion Cyclotomic Extension]{Torsion of Rational Elliptic Curves over the Cyclotomic Extensions of $\mathbb{Q}$}

\author{\"{O}mer Avci}

\address{Dept. of Mathematics, Bogazici University, Istanbul, Turkey}
\email{omeravci372742@gmail.com} 



\begin{abstract} Let $E$ be an elliptic curve defined over $\mathbb{Q}$. In this article, we classify all groups that can arise as $E(\mathbb{Q}(\zeta_p))_{\text{tors}}$ up to isomorphism for any prime $p$. When $p - 1$ is not divisible by small integers such as $3, 4, 5, 7$, or $11$, we obtain a sharper classification. For any abelian number field $K$, the torsion subgroup $E(K)_{\text{tors}}$ is a subgroup of $E(\mathbb{Q}^{\text{ab}})_{\text{tors}}$. Our methods provide tools to eliminate non-realized torsion structures from the list of possibilities for $E(K)_{\text{tors}}$.

\end{abstract}

\maketitle

\section{Introduction and Notation}

Let $K$ be a number field and $E$ an elliptic curve over $K$. By the Mordell–Weil theorem, the group $E(K)$ of $K$-rational points is finitely generated, so that
\begin{equation*}
E(K) \cong E(K)_{\text{tors}} \oplus \mathbb{Z}^{r_K},
\end{equation*}
where $r_K \geq 0$ is a non-negative integer called the \emph{rank}, and $E(K)_{\text{tors}}$ is a finite group called the \emph{torsion subgroup} of $E$ over $K$.

In fact, due to a theorem of Merel~\cite{Merel}, there is a bound on the size of the torsion subgroup that depends only on the degree of $K$ over $\mathbb{Q}$. Thus, there is a finite list of torsion subgroups that can appear as $E(K)_{\text{tors}}$ as $K$ varies over number fields of a fixed degree $d$ and $E/K$ varies. This means that there are only finitely many torsion subgroups (up to isomorphism) that can appear as $E(K)_{\text{tors}}$ when $K$ has a fixed degree and $E$ is defined over $K$.

Determining the possible groups for a fixed degree $d$ has been solved by Mazur for $d = 1$ in~\cite{Mazur}; by Kenku~\cite{kamienny}, Kamienny, and Momose~\cite{kenkumomose} for $d = 2$; and by Derickx, Etropolski, van Hoeij, Morrow, and Zureick-Brown in a joint paper for $d = 3$~\cite{cubicclass}. The case $d = 4$ has recently been resolved by Derickx and Najman in~\cite{derickxnajman}.

Another variant of this problem is to determine the torsion subgroups (up to isomorphism) that can appear as $E(K)_{\text{tors}}$ when $K$ has a fixed degree and $E$ is defined over $\mathbb{Q}$. Clearly, the set of torsion subgroups in this problem is a subset of the set of torsion subgroups in the previous problem. This problem has been solved for $d = 2, 3$ by Najman, for $d = 4$ by Gonz\'{a}lez-Jim\'{e}nez and Najman~\cite{quarticgjn}, for $d = 5$ by Gonz\'{a}lez-Jim\'{e}nez, and for $d = 6$ by Gužvić~\cite{guzvicsextic}.

In our work, we classify $E(\mathbb{Q}(\zeta_p))_{\text{tors}}$ for any prime $p$. By doing so, we also classify 
$E(\mathbb{Q}(\mu_{p^\infty}))_{\text{tors}}$ for those primes. One of our main results is the following theorem. To derive the shortest possible list, we impose the condition that 
$p - 1$ is not divisible by 
$3$, $4$, or $5$. The proof is presented in Section~\ref{eliminationsection}, while Section~\ref{relaxingsection} examines cases with relaxed conditions to include all primes.
\begin{theorem}\label{maintheorem}
    Let $E/\mathbb{Q}$ be an elliptic curve, and let $p > 3$ be a prime such that $p - 1$ is not divisible by $3$, $4$, or $5$. Then $E(\mathbb{Q}(\zeta_p))_{\text{tors}}$ is either one of the groups from Mazur’s theorem, or one of the following groups:
  \begin{equation*}
    \mathbb{Z}/2\mathbb{Z} \times \mathbb{Z}/10\mathbb{Z}, \quad 
    \mathbb{Z}/2\mathbb{Z} \times \mathbb{Z}/12\mathbb{Z}, \quad 
    \text{or} \quad \mathbb{Z}/16\mathbb{Z}.
 \end{equation*}
\end{theorem}

It is also natural to consider the possible torsion subgroups when $K$ is an infinite extension. In this case, the Mordell–Weil theorem does not apply, so it is unknown whether $E(K)_{\text{tors}}$ is finite. However, for certain special infinite extensions $K$, it can be shown that $E(K)_{\text{tors}}$ is finite and takes only finitely many possible forms as $E/\mathbb{Q}$ varies. For example, Fujita classified all possible groups $E(\mathbb{Q}(2^\infty))_{\text{tors}}$ in \cite{Fujita}, where $\mathbb{Q}(2^\infty)$ denotes the compositum of all quadratic fields. Similarly, Daniels, Lozano-Robledo, Najman, and Sutherland classified all possible groups $E(\mathbb{Q}(3^\infty))_{\text{tors}}$ in \cite{MaximalCubic}, where $\mathbb{Q}(3^\infty)$ is the compositum of all cubic fields.

A recent development by Chou \cite{Chou} classified all possible groups $E(\mathbb{Q}^{ab})_{\text{tors}}$, where $\mathbb{Q}^{ab}$ is the maximal abelian extension of $\mathbb{Q}$. Prior to this, Ribet's theorem \cite{Ribet} established the finiteness of $E(\mathbb{Q}^{ab})_{\text{tors}}$ for any elliptic curve $E/\mathbb{Q}$. However, Ribet’s result did not imply that $E(\mathbb{Q}^{ab})_{\text{tors}}$ can only realize finitely many group structures, since Merel’s bound \cite{Merel} does not extend to infinite extensions. Chou’s classification allowed the description of $E(K)_{\text{tors}}$ for any abelian number field $K$ as $E/\mathbb{Q}$ varies, with all possible groups being subgroups of those appearing in Theorem 1.2 of \cite{Chou}. The remaining task is to eliminate from this list the groups that are not realized by any $E/\mathbb{Q}$.

\begin{theorem}[Chou, \cite{Chou}]
  
      Let $E/\mathbb{Q}$ be an elliptic curve. 
    Then $E(\mathbb{Q}^{ab})_{\text{tors}}$ is isomorphic to one of the following
groups
$$ E(\mathbb{Q}^{ab})_{\text{tors}} \simeq
  \begin{cases}
      \mathbb{Z} / N \mathbb{Z} & \text{ with } 1\leq N\leq 19, \text{ or } N=21,25,27,37,43,67,163, \\
           \mathbb{Z} / 2 \mathbb{Z} \times      \mathbb{Z} /2 N \mathbb{Z} & \text{ with } 1\leq N\leq 9,\\
           \mathbb{Z} /3  \mathbb{Z}  \times    \mathbb{Z} /3 N \mathbb{Z} & \text{ with } N=1,2,3, \\
              \mathbb{Z} /4  \mathbb{Z} \times    \mathbb{Z} /4N \mathbb{Z} &\text{ with } N=1,2,3,4, \\
                 \mathbb{Z} /5  \mathbb{Z} \times    \mathbb{Z} /5 \mathbb{Z},\\
                    \mathbb{Z} /6  \mathbb{Z} \times    \mathbb{Z} /6 \mathbb{Z},\\
                       \mathbb{Z} /8  \mathbb{Z} \times    \mathbb{Z} /8 \mathbb{Z}.
  \end{cases}  $$
Each of the groups listed above is realized as  $E(\mathbb{Q}^{ab})_{\text{tors}}$ for some elliptic curve $E/\mathbb{Q}$.
    \end{theorem}

In the result of Gonz\'{a}lez-Jim\'{e}nez and Najman \cite{quarticgjn}, it is shown that if $K$ is a number field whose degree is not divisible by $2, 3, 5,$ or $7$, then for any elliptic curve $E/\mathbb{Q}$ we have
\begin{equation*}
E(K)_{\text{tors}} = E(\mathbb{Q})_{\text{tors}}.
\end{equation*}
We obtain a similar result under the additional assumption that $K/\mathbb{Q}$ is a Galois extension, relaxing the divisibility condition on $2$ to $4$. Similar to their result, we show that many torsion points are either rational points on $E$ or correspond to rational points on a quadratic twist of $E$.
 This allows us to further restrict the possible torsion group structures using Mazur’s theorem.

Let $\mu_n$ denote the set of $n$-th roots of unity for a positive integer $n$, and let $\zeta_p$ be a primitive $p$-th root of unity. For a prime number $p$, define
\begin{equation*}
\mu_{p^\infty} = \{\omega \in \mathbb{C} : \omega^{p^k} = 1 \text{ for some integer } k \geq 0 \}.
\end{equation*}
Note that $\mathbb{Q}(\mu_{p^\infty})$ is the field obtained by adjoining to $\mathbb{Q}$ all primitive $p^n$-th roots of unity for all positive integers $n$.

Gu\v{z}vi\'{c} and Krijan \cite{guzvickrijan} showed that for all primes $p > 3$,
\begin{equation*}
E(\mathbb{Q}(\mu_{p^\infty}))_{\text{tors}} = E(\mathbb{Q}(\zeta_p))_{\text{tors}}.
\end{equation*}
This result enables us to classify $E(\mathbb{Q}(\mu_{p^\infty}))_{\text{tors}}$ for primes $p$ such that $p-1$ is not divisible by $3, 4, 5, 7,$ or $11$.  
Furthermore, Gu\v{z}vi\'{c} and Vukorepa \cite{guzvicvukorepa} classified all possible torsion subgroups appearing as $E(\mathbb{Q}(\zeta_p))_{\text{tors}}$ for $p = 5, 7, 11$.  
Our work extends the classification of $E(\mathbb{Q}(\zeta_p))_{\text{tors}}$ to all primes $p$.  
It is important to note that while our classification eliminates most non-realizable torsion subgroups for all primes, some groups in our list may still not be realized for any specific prime $p$.

Throughout this paper, elliptic curves are identified using their LMFDB labels.

In several parts of our analysis, we rely on computations performed with \texttt{SageMath}. These codes are publicly available on GitHub at \\
\url{https://github.com/omeravci372742/torsion-subgroups-sagemath}.

\vspace{3mm}

\noindent \textbf{Acknowledgements.} 
I would like to thank Ekin Ozman for her guidance
during my master’s thesis, from which this paper 
has grown. I am also grateful to Álvaro Lozano-Robledo for his insights on the connection between
ramified primes and the field of definition of torsion points. Special thanks to Filip Najman for
his helpful discussions, comments, and for 
bringing related work by his students to my 
attention. I would also like to express my 
appreciation to the anonymous referee for their 
valuable suggestions, particularly the connection 
between the conductor and quadratic twists, and 
the insightful idea of using division polynomials 
to handle all twists simultaneously. I am additionally 
thankful to the referee for pointing out the method of 
factoring over local fields to obtain faster results.
Finally, I am 
deeply thankful to Sueda Senturk Avci for her unwavering 
support and assistance with computations, 
especially those involving \texttt{SageMath}.

\section{Torsion of Rational Elliptic Curves}

In this section, we will present some theorems about the classification of $E(K)_{\text{tors}}$.

\begin{theorem}[Mazur, \cite{Mazur}]
   Let $E/\mathbb{Q}$ be an elliptic curve. Then
  $$E(\mathbb{Q})_{\text{tors}} \simeq 
  \begin{cases}
      \mathbb{Z} / N \mathbb{Z} & \text{ with } 1\leq N\leq 10 \text{ or } N=12 \text{ or},  \\
           \mathbb{Z} / 2 \mathbb{Z} \times      \mathbb{Z} /2 N \mathbb{Z} & \text{ with } 1\leq N\leq 4.
  \end{cases}  $$ 
  All of the 15 torsion subgroups appear infinitely often.
\end{theorem}

\begin{theorem}[Najman, \cite{NajmanQuadratic}]
\label{najmanquadratic}
  Let $E/\mathbb{Q}$ be an elliptic curve defined over $\mathbb{Q}$, and let $K$
be a quadratic number field. Then
$$E(K)_{\text{tors}} \simeq 
  \begin{cases}
      \mathbb{Z} / N \mathbb{Z} & \text{ with } 1\leq N\leq 10 \text{ or } N=12,15,16,\text{ or}\\
           \mathbb{Z} / 2 \mathbb{Z} \times      \mathbb{Z} /2 N \mathbb{Z} & \text{ with } 1\leq N\leq 6 \text{ or}\\
           \mathbb{Z} /3  \mathbb{Z}  \times    \mathbb{Z} /3 N \mathbb{Z} & \text{ with } N=1,2, \text{ only if } K=\mathbb{Q}(\sqrt{-3}) \text{ or}\\
              \mathbb{Z} /4  \mathbb{Z} \times    \mathbb{Z} /4 \mathbb{Z} & \text{ only if } K=\mathbb{Q}(\sqrt{-1}).        
  \end{cases}  $$
  Each of these groups, except for $\mathbb{Z}/15\mathbb{Z}$, appears as the torsion structure over a quadratic field for infinitely many rational elliptic curves $E$. The elliptic curves \lmfdbec{50}{b}{3} and \lmfdbec{50}{a}{2} have $15$-torsion over $\mathbb{Q}(\sqrt{5})$, \lmfdbec{50}{b}{4} and \lmfdbec{450}{g}{4} have $15$-torsion over $\mathbb{Q}(\sqrt{-15})$. These are the only
rational curves having non-trivial $15$-torsion over any quadratic field.
\end{theorem}

\begin{theorem}[Chou, \cite{Chou}, Corollary to Theorem 1.2]\label{allgroupslist}
      Let $E/\mathbb{Q}$ be an elliptic curve. Let $K$ be an abelian number field.
    Then $E(K)_{\text{tors}}$ is isomorphic to one of the following
groups:
$$E(K)_{\text{tors}} \simeq 
  \begin{cases}
      \mathbb{Z} / N \mathbb{Z} & \text{ with } 1\leq N\leq 19, \text{ or } N=21,25,27,37,43,67,163, \\
           \mathbb{Z} / 2 \mathbb{Z} \times      \mathbb{Z} /2 N \mathbb{Z} & \text{ with } 1\leq N\leq 9,\\
           \mathbb{Z} /3  \mathbb{Z}  \times    \mathbb{Z} /3 N \mathbb{Z} & \text{ with } N=1,2,3, \\
              \mathbb{Z} /4  \mathbb{Z} \times    \mathbb{Z} /4N \mathbb{Z} &\text{ with } N=1,2,3,4, \\
                 \mathbb{Z} /5  \mathbb{Z} \times    \mathbb{Z} /5 \mathbb{Z},\\
                    \mathbb{Z} /6  \mathbb{Z} \times    \mathbb{Z} /6 \mathbb{Z},\\
                       \mathbb{Z} /8  \mathbb{Z} \times    \mathbb{Z} /8 \mathbb{Z}.
  \end{cases}  $$
\end{theorem}

  \begin{lemma}[Chou, \cite{Chouquartic}]\label{veryuseful}
    Let $E$ be an elliptic curve defined over $K$, $\alpha$ square-free in $K$.
If $n$ is odd, then there exists an isomorphism
\begin{equation*}E(K(\sqrt{\alpha}))[n]\cong E(K)[n] \oplus E_\alpha(K)[n] \end{equation*}
where $E_\alpha$ denotes the $\alpha$-quadratic twist of $E$.
\end{lemma}

\begin{theorem}[Lozano-Robledo, \cite{fieldofdefinition}]\label{FieldofDefinition}
    Let $p\geq 11$, be a prime other than $13$ or $37$. Let $E/\Q$ be an elliptic curve and $P$ be a point of order $p$ on $E$. Then, $[\Q(P):\Q]\geq (p-1)/2$ holds, where $\Q(P)=\Q(x(P),y(P))$ is the extension of rational numbers by the coordinates of the point $P$ called the field of definition of $P$.
\end{theorem}
It is important to emphasize that the field of definition $\mathbb{Q}(P)$ of a point $P$ of order $p$ on an elliptic curve $E$ depends on the choice of $P$. For distinct points $P$ and $Q$ of order $p$, the fields $\mathbb{Q}(P)$ and $\mathbb{Q}(Q)$ may differ, both being subfields of $\mathbb{Q}(E[p])$, the field of definition of all $p$-torsion points on $E$.

In this paper, the \emph{field of definition of the $p$-torsion} refers to the field of smallest degree over $\mathbb{Q}$. When focusing on abelian Galois groups, we specify the smallest such field with abelian Galois group, particularly in the proof of Theorem \ref{mostmaintheorem} and Section \ref{realizationsection}, where we study the realization of torsion subgroups. This distinction is essential for clarity in our classification and understanding of torsion subgroup realizations over various fields.

\section{Eliminating Possible Torsion}\label{eliminationsection}

In this section, we present methods developed to eliminate possible torsion subgroups from the list in Theorem \ref{allgroupslist} that can appear as $E(K)_{\text{tors}}$, where $E/\mathbb{Q}$ is an elliptic curve and $K$ is an abelian number field.

\begin{lemma} \label{4pairinggeneralmore}
Let $K$ be a Galois number field of degree $n$, and let $E/\mathbb{Q}$ be an elliptic curve. Let $q$ be an odd prime power, and suppose 
\begin{equation*}
E(K)[q] \cong \mathbb{Z}/q\mathbb{Z}.
\end{equation*}
    \begin{enumerate}[(i)]

    \item If $\gcd(n, \varphi(q)) = 1$, then 
    \begin{equation*}
    E(K)[q] = E(\mathbb{Q})[q].
    \end{equation*}
        
    \item If $\gcd(n, \varphi(q)) = 2$, then there exists a square-free integer $d$ with $\sqrt{d} \in K$ such that
    \begin{equation*}
    E(K)[q] = E(\mathbb{Q}(\sqrt{d}))[q] \cong E_d(\Q)[q].
    \end{equation*}
        \end{enumerate}

\end{lemma}
\begin{proof}
Without loss of generality, we assume $E$ is given by the short Weierstrass equation
\begin{equation*}
E : y^2 = x^3 + A x + B
\end{equation*}
for some $A, B \in \mathbb{Q}$. 

Let $P \in E(K)[q]$ be a point of order $q$. Let $\sigma \in \Gal(K/\mathbb{Q})$. Then
\begin{equation*}
P^\sigma = a P \quad \text{for some } a \in (\mathbb{Z}/q \mathbb{Z})^\times
\end{equation*}
 because $P^\sigma \in E(K)[q]$. 
 
 Applying $\sigma$ to $P$ exactly $n$ times, we get $P^{\sigma^n} = P$, since $\sigma^n = \mathrm{id}$. Similarly, applying $\sigma$ to $P$ exactly $\phi(q)$ times, we get $P^{\sigma^{\phi(q)}} = P$ because $a^{\phi(q)} \equiv 1 \pmod{q}$.

If $\gcd(n, \phi(q)) = 1$, then by the Euclidean algorithm we get
\begin{equation*}
P^\sigma = P \quad\text{for all } \sigma \in \Gal(K/\mathbb{Q}),
\end{equation*}
 so $P$ is defined over $\mathbb{Q}$. Since $P$ generates $E(K)[q]$, it follows that
\begin{equation*}
E(K)[q] = E(\mathbb{Q})[q].
\end{equation*}

If $\gcd(n, \phi(q)) = 2$, then by the Euclidean algorithm,
\begin{equation*}
P^{\sigma^2} = P \quad \text{for all } \sigma\in \Gal(K/\Q),
\end{equation*}
 which implies
\begin{equation*}
a \equiv \pm 1 \pmod{\phi(q)}
\quad \text{and} \quad
P^\sigma = \pm P.
\end{equation*}
Writing $P = (x_0,y_0)$, this means
\begin{equation*}
(x_0^\sigma, y_0^\sigma) = \pm (x_0,y_0) = (x_0, \pm y_0).
\end{equation*}
Therefore, $x_0^\sigma = x_0$
for all $\sigma \in \Gal(K/\Q)$, so $x_0$ is rational. This implies $y_0^2$ is rational, and hence
\begin{equation*}
y_0 \in \mathbb{Q}(\sqrt{d}) \quad \text{and} \quad y_0\sqrt{d} \in \Q
\end{equation*}
for some square-free integer $d$ with $\sqrt{d} \in K$. Thus,
\begin{equation*}
P \in E(\mathbb{Q}(\sqrt{d}))[q]
\end{equation*}
and $P$ generates $E(\mathbb{Q}(\sqrt{d}))[q]$, so
\begin{equation*}
E(K)[q] = E(\mathbb{Q}(\sqrt{d}))[q].
\end{equation*}
When we use the quadratic twist model 
\begin{equation*}
    E_d : d y^2 = x^3 + A x + B,
\end{equation*}
the image of $P$ on $E_d$ becomes the point $P_d = (x, y / \sqrt{d})$. Then $P_d \in E_d(\mathbb{Q})_{\text{tors}}$. We also have
\begin{equation*}
    E(\mathbb{Q}(\sqrt{d}))[q] \cong E_d(\mathbb{Q}(\sqrt{d}))[q].
\end{equation*}
Therefore, $E_d(\mathbb{Q}(\sqrt{d}))[q]$ is generated by the point $P_d$, and so
\begin{equation*}
    E_d(\mathbb{Q}(\sqrt{d}))[q] = E_d(\mathbb{Q})[q].
\end{equation*}
Ultimately, for the square-free integer $d$ satisfying $\sqrt{d} \in K$, we obtain
\begin{equation*}
    E(K)[q] = E(\mathbb{Q}(\sqrt{d}))[q] \cong E_d(\mathbb{Q}(\sqrt{d}))[q] = E_d(\mathbb{Q})[q].
\end{equation*}
This completes the proof by showing that the $q$-torsion is defined over $\mathbb{Q}$ after a suitable twist.
\end{proof}

\begin{lemma}\label{mainlemma}
         Let $K$ be a finite Galois extension of $\mathbb{Q}$ with $[K:\mathbb{Q}]=n$ for some positive integer $n$. Let $E/\mathbb{Q}$ be an elliptic curve. Then,
       \begin{enumerate}[(i)]
           \item if the extension degree $n$ is not divisible by $3$, then $E(K)_{\text{tors}}\not \cong \mathbb{Z}/N\mathbb{Z}$ for $N=14,18,19,21,27,163$, and $E(K)_{\text{tors}}\not \cong \mathbb{Z}/2\mathbb{Z} \times \mathbb{Z}/2N\mathbb{Z}$ for $N=7,9$,
           \item if the extension degree $n$ is not divisible by $3$ and $4$, then $E(K)_{\text{tors}}\not \cong \mathbb{Z}/N\mathbb{Z}$ for $N=13,37$,
           \item  if the extension degree $n$ is not divisible by $3$ and $7$, then $E(K)_{\text{tors}}\not \cong \mathbb{Z}/43\mathbb{Z}$,
            \item  if the extension degree $n$ is not divisible by $3$ and $11$, then $E(K)_{\text{tors}}\not \cong \mathbb{Z}/67\mathbb{Z}$.
            
              \item  if the extension degree $n$ is not divisible by $4$, then $E(K)_{\text{tors}}\not \cong \mathbb{Z}/17\mathbb{Z}$.

     \item  if the extension degree $n$ is not divisible by $4$, and the fields $\mathbb{Q}(\sqrt{5})$ and $\mathbb{Q}(\sqrt{-15})$ are not subfields of $K$, then $E(K)_{\text{tors}}\not \cong \mathbb{Z}/15\mathbb{Z}$.
     
     \item if the extension degree $n$ is not divisible by $5$, then $E(K)_{\text{tors}}\not \cong \mathbb{Z}/11\mathbb{Z}$.
     
     \item if the extension degree $n$ is not divisible by $4$ and $5$ then $E(K)_{\text{tors}}\not \cong \mathbb{Z}/25\mathbb{Z}$.

       \end{enumerate}
\end{lemma}

\begin{proof}
   Let us show that $E(K)_{\text{tors}} \cong \mathbb{Z}/N\mathbb{Z}$ is not possible for the values of $N$ other than $14, 15, 18, 21$ given in the lemma, under the appropriate assumptions on the divisors of $n$. All such values of $N$ are either prime or prime powers, so we can apply Lemma~\ref{4pairinggeneralmore}. Under the assumptions on $n$, we have $\gcd(n, \phi(N)) \leq 2$, and hence $E(K)_{\text{tors}} \cong E_d(\mathbb{Q})_{\text{tors}}$ for some square-free integer $d$ with $\sqrt{d} \in K$. Since $E_d$ is also an elliptic curve over $\mathbb{Q}$, Mazur's theorem implies that $E_d(\mathbb{Q})_{\text{tors}} \cong \mathbb{Z}/N\mathbb{Z}$ is not possible.

    Showing that $E(K)_{\text{tors}} \cong \mathbb{Z}/N\mathbb{Z}$ is not possible for $N = 14$ and $N = 18$ can be done similarly, but it requires an additional step. Assume the contrary. Then $E(K)_{\text{tors}}$ has a single $2$-torsion point, which implies that there is a rational $2$-torsion point in $E(K)_{\text{tors}}$. Consequently, for any square-free integer $d$, the torsion subgroup $E_d(K)_{\text{tors}}$ also has a rational $2$-torsion point. 

If $n$ is not divisible by $3$, then by Lemma~\ref{4pairinggeneralmore}, the torsion subgroup $E_d(\mathbb{Q})_{\text{tors}}$ has a point of order $7$ or $9$, respectively, for some square-free integer $d$ with $\sqrt{d} \in K$. Since $E_d(\mathbb{Q})_{\text{tors}}$ contains both a $2$-torsion point and a point of order $7$ or $9$, it must contain a point of order $14$ or $18$, respectively. This contradicts Mazur's theorem.

   The case $E(K)_{\text{tors}} \cong \mathbb{Z}/21\mathbb{Z}$ requires more careful consideration. By Lemma~\ref{4pairinggeneralmore}, we obtain 
\begin{equation*}
E(K)[3] \cong E_{d_1}(\mathbb{Q})[3] \quad \text{and} \quad E(K)[7] \cong E_{d_2}(\mathbb{Q})[7]    
\end{equation*}
for some square-free integers $d_1$ and $d_2$ satisfying $\sqrt{d_1}, \sqrt{d_2} \in K$. Clearly, $E_{d_1}(\mathbb{Q})[3]$ is a subgroup of $E_{d_1}(\mathbb{Q})[21]$, and similarly $E_{d_2}(\mathbb{Q})[7]$ is a subgroup of $E_{d_2}(\mathbb{Q})[21]$.

If $d_1 = d_2$, then we obtain $E_{d_1}(\mathbb{Q})[21] \cong \mathbb{Z}/21\mathbb{Z}$, which contradicts Mazur's theorem, since $E_{d_1}$ is defined over $\mathbb{Q}$. 

Otherwise, applying Lemma~\ref{veryuseful}, we obtain
\begin{equation*}
E_{d_1}(\mathbb{Q}(\sqrt{d_1 d_2}))[21] \cong E_{d_1}(\mathbb{Q})[21] \oplus E_{d_2}(\mathbb{Q})[21].
\end{equation*}
Note that $E_{d_1}(\mathbb{Q})[3] \cong \mathbb{Z}/3\mathbb{Z}$ and $E_{d_2}(\mathbb{Q})[7] \cong \mathbb{Z}/7\mathbb{Z}$ since $E(K)_{\text{tors}} \cong \mathbb{Z}/21\mathbb{Z}$. Therefore, the group $E_{d_1}(\mathbb{Q}(\sqrt{d_1 d_2}))[21]$ has a subgroup isomorphic to $\mathbb{Z}/21\mathbb{Z}$. In particular, this group is contained in the torsion subgroup $E_{d_1}(\mathbb{Q}(\sqrt{d_1 d_2}))_{\text{tors}}$.

Since $E_{d_1}$ is defined over $\mathbb{Q}$ and we are considering its torsion over a quadratic field, Theorem~\ref{najmanquadratic} implies that it cannot have a subgroup isomorphic to $\mathbb{Z}/21\mathbb{Z}$. Hence, $E(K)_{\text{tors}} \cong \mathbb{Z}/21\mathbb{Z}$ is not possible.

   Similarly, the case $E(K)_{\text{tors}} \cong \mathbb{Z}/15\mathbb{Z}$ also requires special attention. By applying the same arguments, we find that $E_{d_1}(\mathbb{Q}(\sqrt{d_1 d_2}))_{\text{tors}}$ has a subgroup of order $15$ for some square-free integers $d_1, d_2$ such that $\sqrt{d_1}, \sqrt{d_2} \in K$. Since $E_{d_1}$ is defined over $\mathbb{Q}$, Theorem~\ref{najmanquadratic} implies that $d_1d_2 = 5$ or $d_1d_2 = -15$. However, this is impossible because we assumed both $\sqrt{5}$ and $\sqrt{-15}$ do not belong to the field $K$.

   Finally, we show that $E(K)_{\text{tors}} \cong \mathbb{Z}/2\mathbb{Z} \times \mathbb{Z}/2N\mathbb{Z}$ is not possible for $N=7,9$. Suppose otherwise. Then, by Lemma~\ref{4pairinggeneralmore}, we have 
\begin{equation*}
E(K)[N] \cong E_d(\mathbb{Q})[N]
\end{equation*}   
for some square-free integer $d$ with $\sqrt{d} \in K$. Since $E(K)$ has full $2$-torsion, $E_d(K)$ also has full $2$-torsion. If the cubic polynomial $x^3 + A x + B$ is irreducible over $\mathbb{Q}$, then the degree $n = [K:\mathbb{Q}]$ must be divisible by $3$, contradicting our assumption. Therefore, the cubic must have a rational root, implying that both $E$ and $E_d$ have a rational $2$-torsion point. Recall that 
\begin{equation*}
E(K)[N] \cong E_d(\mathbb{Q})[N] \cong \mathbb{Z}/N\mathbb{Z}.
\end{equation*}
Hence, $E_d(\mathbb{Q})$ must contain a subgroup of order $2N$, contradicting Mazur's theorem.
\end{proof}

\begin{lemma}\label{mod16}
      Let $K$ be a finite Galois extension of $\mathbb{Q}$ with $[K:\mathbb{Q}]=n$ for some positive integer $n$. Let $E/\mathbb{Q}$ be an elliptic curve. If the extension degree $n$ is not divisible by $4$, then $E(K)_{\text{tors}}\not \cong \mathbb{Z}/2\mathbb{Z} \times \mathbb{Z}/16\mathbb{Z}$.
\end{lemma}
\begin{proof}
Assume, for the sake of contradiction, that there exists a point 
\begin{equation*}
P = (x_0, y_0) \in E(K)_{\text{tors}}
\end{equation*}
of order $16$. For any $\sigma \in \Gal(K/\mathbb{Q})$, the point $P^\sigma\in E(K)_{\text{tors}}$ also has order $16$. Since there are exactly $16$ points of order $16$ in $E(K)_{\text{tors}}$, their $x$-coordinates form a set of $8$ elements, assuming $E$ is given in short Weierstrass form.

Consider the orbit of $x_0$ under the Galois action:
\begin{equation*}
\{ x_0^\sigma : \sigma \in \Gal(K/\mathbb{Q}) \}.
\end{equation*}
This set corresponds to the Galois conjugates of $x_0$, and its size equals the degree of the field extension $\mathbb{Q}(x_0)/\mathbb{Q}$. Since $\mathbb{Q}(x_0) \subseteq K$ and $[K : \mathbb{Q}] = n$, the degree $[\mathbb{Q}(x_0) : \mathbb{Q}]$ divides $n$. By assumption, $n$ is not divisible by $4$, so the size of the orbit cannot be $4$ or $8$.

Consider the partitions of $8$ elements into subsets. Either there exists a subset with at most $2$ elements, or the partition consists of two subsets of sizes $3$ and $5$. We will show at the end of the proof that the latter case is not possible.

Without loss of generality, assume that the field extension $\mathbb{Q}(x_0)$ has degree at most $2$, where $P = (x_0, y_0)$ is a point of order $16$. If the orbit size is at most $2$, then
$\mathbb{Q}(x_0) \subseteq \mathbb{Q}(\sqrt{d})$
for some square-free integer $d$. 
   Since 
\begin{equation*}
[\mathbb{Q}(x_0,y_0) : \mathbb{Q}(x_0)] \leq 2,
\end{equation*}
the degree of the field extension $\mathbb{Q}(P) = \mathbb{Q}(x_0, y_0)$ divides $4$. Moreover, since the degree also divides $n$, we conclude that $[\mathbb{Q}(P) : \mathbb{Q}] \leq 2$.

Now consider the extension $\mathbb{Q}(E[2])$, obtained by adjoining the coordinates of all $2$-torsion points of $E$. Assuming $E$ is given in short Weierstrass form, $\mathbb{Q}(E[2])$ is the splitting field of the cubic polynomial
\begin{equation*}
x^3 + A x + B,
\end{equation*}
with $A,B \in \mathbb{Q}$. The point $8P$ is a $2$-torsion point and lies in the field $\mathbb{Q}(P)$. Therefore, one root of the cubic must lie in $\mathbb{Q}(P)$, which is at most quadratic over $\mathbb{Q}$. This implies that the cubic is not irreducible over $\mathbb{Q}$, and hence its splitting field, which is $\mathbb{Q}(E[2])$, has degree at most $2$ over $\mathbb{Q}$.

We have thus shown that there exists a point $P \in E(K)_{\text{tors}}$ of order $16$ such that $\mathbb{Q}(P)$ has degree at most $2$, and that $\mathbb{Q}(E[2])$ also has degree at most $2$. Therefore, the compositum $\mathbb{Q}(E[2], P)$ is an extension of degree dividing $4$, and since it also divides $n$, its degree is at most $2$. This proves that all $2$-torsion points and the point $P$ lie in a quadratic subfield $\mathbb{Q}(\sqrt{d})$ of $K$, for some square-free integer $d$ (possibly $d = 1$).

Finally, since all points in $E(K)_{\text{tors}}$ are generated by the $2$-torsion points and a point of order $16$, we conclude that
\begin{equation*}
E(K)_{\text{tors}} = E(\mathbb{Q}(\sqrt{d}))_{\text{tors}},
\end{equation*}
which contradicts Theorem \ref{najmanquadratic}.

To conclude our proof, we will show that the set of $x$-coordinates of the $16$-torsion points cannot be partitioned into two Galois orbits of sizes $3$ and $5$ under the action of $\Gal(K/\mathbb{Q})$.

To obtain a contradiction, assume that $P, Q \in E(K)_{\text{tors}}$ are points of order $16$ such that the Galois orbits $\{x_0^\sigma : \sigma \in \Gal(K/\mathbb{Q})\}$ and $\{x_1^\sigma : \sigma \in \Gal(K/\mathbb{Q})\}$ contain exactly $3$ and $5$ elements, respectively, where $P = (x_0, y_0)$ and $Q = (x_1, y_1)$.

This means that $\mathbb{Q}(x_0)$ is a cubic field and that $\mathbb{Q}(P)$ has extension degree $3$ or $6$. The field $\mathbb{Q}(E[2])$ is the splitting field of the cubic $x^3 + Ax + B$, so its degree over $\mathbb{Q}$ is $1$, $2$, $3$, or $6$. Moreover, since all of $E(K)_{\text{tors}}$ is generated by $P$ and the $2$-torsion points, we have 
\begin{equation*}
 E(K)_{\text{tors}} = E(L)_{\text{tors}}
\end{equation*}
where we define $L = \mathbb{Q}(E[2], P) \subset K$.

Since $n$ is not divisible by $4$, the degree $[\mathbb{Q}(L):\mathbb{Q}]$ is also not divisible by $4$. Moreover, it must divide $36$. Hence, we must have
\begin{equation*}
    [\Q(L):\Q]\in\{1,2,3,6,9\}.
\end{equation*}
However, since $Q \in E(L)_{\text{tors}}$, it follows that $\mathbb{Q}(x_1) \subseteq \mathbb{Q}(Q) \subseteq L$, which contradicts the assumption that $[\mathbb{Q}(x_1):\mathbb{Q}] = 5$.
\end{proof}

Now, we will proceed to the proof of our first theorem.

\begin{proof}[Proof of Theorem \ref{maintheorem}]
Since $\mathbb{Q}(\zeta_p)$ is an abelian number field of degree $p-1$, we know that the torsion subgroup
$E(\mathbb{Q}(\zeta_p))_{\text{tors}}$
is isomorphic to one of the groups given in Theorem \ref{allgroupslist}. Recall that the groups listed there are all possible subgroups of the groups that can be realized as 
$E(\mathbb{Q}^{ab})_{\text{tors}}$.
In our proof, we will show that most of those groups are not realizable.

Firstly, let us prove that 
$\mathbb{Z}/N\mathbb{Z} \times \mathbb{Z}/N\mathbb{Z}$
cannot be a subgroup of 
$E(\mathbb{Q}(\zeta_p))_{\text{tors}}$
when $N=3,4,5,6,8$. Due to the Weil pairing, if 
$\mathbb{Z}/N\mathbb{Z} \times \mathbb{Z}/N\mathbb{Z}$
is a subgroup of 
$E(\mathbb{Q}(\zeta_p))_{\text{tors}}$,
then 
\begin{equation*}
\zeta_N \in \mathbb{Q}(\zeta_p).
\end{equation*}
This is not the case for $N=3,4,5,6,8,$ since $p \geq 11$ is a prime.

As a combined result of Lemma \ref{4pairinggeneralmore}, by assuming that $p-1$ is not divisible by $3,4,5,7,11,$ we can use all items of the lemma except the result about $\mathbb{Z}/15\mathbb{Z}$. This result requires more conditions. We know that when $p$ is a prime, $\mathbb{Q}(\zeta_p)$ has a unique quadratic subfield. If 
$p \equiv 1 \pmod{4}$,
that unique quadratic subfield is 
$\mathbb{Q}(\sqrt{p})$.
If 
$p \equiv 3 \pmod{4}$,
as in our theorem, the unique quadratic subfield is 
$\mathbb{Q}(\sqrt{-p})$.
This means that $\mathbb{Q}(\sqrt{5})$ and $\mathbb{Q}(\sqrt{-15})$ are not subfields of $\mathbb{Q}(\zeta_p)$, so we can use that result of the lemma too.

Then, we have shown that 
\begin{equation*}
E(\mathbb{Q}(\zeta_p))_{\text{tors}} \not\cong \mathbb{Z}/N\mathbb{Z} \quad \text{for} \quad 
N = 11,13,15,17,19,21,25,27,37,43,67,163,
\end{equation*}
under the additional assumptions that $p-1$ is not divisible by $7$ or $11$. Next, we demonstrate how these additional assumptions can be removed.

 If we remove the conditions that $p-1$ is not divisible by $7$ and $11$, then we cannot exclude the possibility of 
$\mathbb{Z}/43\mathbb{Z}$ or $\mathbb{Z}/67\mathbb{Z}$
using Lemma \ref{4pairinggeneralmore} alone. Thus, we need a different strategy. Let 
\begin{equation*}
E(\mathbb{Q}(\zeta_p))_{\text{tors}} \cong \mathbb{Z}/\ell\mathbb{Z} \quad \text{ for } \ell = 43 \text{ or } \ell = 67.
\end{equation*}
By Lemma \ref{43isogeny}, this is only possible when 
$p = 43$ and $p = 67$,
respectively, but both cases are excluded by our assumption that $p-1$ is not divisible by $3$.

There is also another way to see that 
$\mathbb{Z}/43\mathbb{Z}$ or $\mathbb{Z}/67\mathbb{Z}$
cannot be realized over $\mathbb{Q}(\zeta_p)$ under the assumption that $p-1$ is not divisible by $3$. We use a similar idea to the proof of Lemma \ref{mainlemma}. Let 
$P \in E(\mathbb{Q}(\zeta_p))_{\text{tors}}$
be a point of order $\ell$ for $\ell=43$ or $\ell=67$. For any 
$\sigma \in \Gal(\mathbb{Q}(\zeta_p)/\mathbb{Q})$,
consider the Galois action. Then 
$\sigma^{(\ell-1)/3}$
fixes any $\ell$-order point in $E(\mathbb{Q}(\zeta_p))_{\text{tors}}$ because 
\begin{equation*}
\gcd(\ell-1, p-1) \mid \frac{\ell-1}{3}.
\end{equation*}

Combined with the fact that 
$\Gal(\mathbb{Q}(\zeta_p)/\mathbb{Q}) \cong \mathbb{Z}/(p-1)\mathbb{Z}$,
this implies that the field of definition of the $\ell$-torsion, 
$\mathbb{Q}(P) \subset \mathbb{Q}(\zeta_p)$,
is an extension of degree at most 
$(\ell-1)/{3}$.
This contradicts Theorem \ref{FieldofDefinition} because the extension degree is not sufficiently large.

By combining Lemma \ref{mainlemma} and Lemma \ref{mod16}, we can show that
\begin{equation*}
E(\mathbb{Q}(\zeta_p))_{\text{tors}} \not\cong \mathbb{Z}/2\mathbb{Z} \times \mathbb{Z}/2N\mathbb{Z}\quad \text{for}\quad
N = 7,8,9.
\end{equation*}
Only remaining possibilities for 
$E(\mathbb{Q}(\zeta_p))_{\text{tors}}$
are the groups in Mazur's theorem and the three remaining groups given in the statement of Theorem \ref{maintheorem}.
\end{proof}

\section{Isogenies}\label{isogenysection}
In this section, we study the relationship between the rational isogenies of an elliptic curve 
$E/\mathbb{Q}$ and its torsion subgroups $E(K)_{\text{tors}}$, where $K$ is a Galois number field. Throughout the rest of the paper, when we say that $E$ has an $n$-isogeny, we mean that there exists a $\mathbb{Q}$-rational isogeny 
$\varphi: E \to E'$
with its kernel isomorphic to $\mathbb{Z}/n\mathbb{Z}$.

The classification of $\mathbb{Q}$-rational isogenies has been carried out by several authors by studying the rational points on the modular curves 
$X_0(n)$.
When the genus of the modular curve $X_0(n)$ is greater than zero, it is well-known that $X_0(n)$ has only finitely many rational points. Some of those points are cusps, and the non-cuspidal points correspond to the $j$-invariants of elliptic curves with an $n$-isogeny.

The explicit list of $j$-invariants can be found in Table 2 of Section 7 in \cite{Chou} and Table 4 of Section 9 in \cite{fieldofdefinition}. For the elliptic curves corresponding to these $j$-invariants, we will use their LMFDB labels. The information on these curves, including torsion subgroups and isogeny structure, is available on the LMFDB database \cite{lmfdb}.

\begin{theorem}[Fricke, Kenku, Klein, Kubert, Ligozat, Mazur, and Ogg, among others] 
If $E/\Q$ has an $n$-isogeny, $n \leq 19 $ or $n \in \{21, 25, 27, 37, 43, 67, 163\}$. If E does not have complex multiplication, then $n \leq 18$ or $n \in \{21, 25, 37\}$.
\end{theorem}

For $n \in \{11, 15, 17, 19, 21, 27, 37, 43, 67, 163\}$, the modular curve $X_0(n)$ has finitely many rational points. Consequently, there are only finitely many elliptic curves, or more precisely, finitely many $j$-invariants, with an $n$-isogeny. The corresponding list of $j$-invariants can be found in Table 2 of Section 7 in \cite{Chou} and Table 4 of Section 9 in \cite{fieldofdefinition}.

\begin{theorem}[Kenku, Theorem 2, \cite{kenku}]\label{Kenku}
There are at most eight $\Q$-isomorphism classes of elliptic curves in each $\Q$ isogeny class.

Let $C_{p}(E)$ denote the number of distinct $\Q$-rational cyclic subgroups of order $p^{n}$ for any $n$ of $E$. Then, we have the following table for bounds on $C_{p}$ for any elliptic curve over $\Q$
$$
\begin{array}{|c|ccccccccccccc|}
\hline p & 2 & 3 & 5 & 7 & 11 & 13 & 17 & 19 & 37 & 43 & 67 & 163 & \text{else} \\ \hline
 C_{p} & 8 & 4 & 3 & 2 & 2 & 2 & 2 & 2 & 2 & 2 & 2 & 2 & 1 \\ \hline
\end{array}
$$

In particular, fix a $\Q$-isogeny class and a representative $E$ of that class.  

\begin{itemize}

\item If $C_{p}(E) = 2$ for some prime $p \geq 11$, then $C_{q}(E)=1$ for all other primes.  So $C(E)=2$.

\item If $C_7(E)=2$, then $C_5(E)=1$ and either $C_3(E) \leq 2$ and $C_2(E) = 1$ or $C_3(E)=1$ and $C_2(E) \leq 2$.  All these yield $C(E) \leq 4$.

\item If $C_5(E)=3$, then $C_p(E)=1$ for all primes $p \neq 5$.

\item If $C_5(E)=2$, then either $C_3(E) \leq 2$ and $C_2(E)=1$ or $C_3(E)=1$ and $C_2(E) \leq 2$.  Hence $C(E) \leq 4$.

\item If $C_3(E)=4$, then there exists a representative of the class of $E$ with a $\Q$-rational cyclic subgroup of order $27$, and $C_2(E)=1$ so $C(E) \leq 4$.

\item If $C_3(E)=3$, then $C_2(E) \leq 2$ so that $C(E) \leq 6$.

\item If $C_3(E) \leq 2$, then $C_2(E) \leq 4$ so that $C(E) \leq 8$.

\end{itemize}

Note the fact that $C(E)=8$ is possible only if $C_2(E)=8$ or $C_3(E)=2$ and $C_2(E)=4$.

\end{theorem}

\begin{lemma}[Chou, Lemma 2.7, \cite{Chou}]\label{n-isogeny}
    Let $K$ be a Galois extension of $\Q$, and let $E$ be an elliptic curve over $\Q$. If $E(K)_{\text{tors}}\cong \Z/m\Z \times \Z/mn\Z$, then $E$ has an $n$-isogeny over $\Q$.
\end{lemma}

\begin{lemma}\label{isogeny-fieldofdefinition}
If $E/\Q$ has a rational $n$-isogeny defined over $\Q$, then, there is a Galois extension $K/\Q$ such that $E(K)_{\text{tors}}$ has a point of order $n$ and $\Gal(K/\Q)$ is isomorphic to a subgroup of the multiplicative group $(\Z/n\Z)^\times$.
\end{lemma}

\begin{proof}
  Let $P = (x_0,y_0) \in E(\overline{\mathbb{Q}})_{\text{tors}}$ be a point of order $n$ such that $\langle P \rangle$ is stable under the action of $\Gal(\overline{\mathbb{Q}}/\mathbb{Q})$. Then for any $\sigma \in \Gal(\overline{\mathbb{Q}}/\mathbb{Q})$, we have
\begin{equation*}
    P^\sigma = aP \quad \text{for some} \quad a \in (\mathbb{Z}/n\mathbb{Z})^\times.
\end{equation*}

Put $K = \mathbb{Q}(P) = \mathbb{Q}(x_0,y_0)$. Then it is clear that $K/\mathbb{Q}$ is a Galois extension because $P^\sigma = aP \in \mathbb{Q}(P)$ for all $\sigma$. Moreover, there is an injective homomorphism
\begin{equation*}
    \rho: \Gal(K/\mathbb{Q}) \to (\mathbb{Z}/n\mathbb{Z})^\times.
\end{equation*}

This shows that $\Gal(K/\mathbb{Q})$ is isomorphic to a subgroup of $(\mathbb{Z}/n\mathbb{Z})^\times$.
\end{proof}

\begin{theorem}[Diamond, Shurman, Theorem 9.4.1, \cite{DiamondShurman}]
    Let $\ell$ be prime and let $E$ be an elliptic curve over $\Q$ with conductor $N_E$. The Galois representation $\rho_{E,\ell}$ is unramified at every prime $p\nmid \ell N_E.$ 

\end{theorem}

\begin{corollary}\label{unramifiedcorollary}
    Let $\ell$ be prime and let $E$ be an elliptic curve over $\Q$ with conductor $N_E$. Then, $\Q(E[\ell])$ is unramified at  every prime $p\nmid \ell N_E.$ 
\end{corollary}
\begin{proof}
Let $p \nmid \ell N$, and let $\mathfrak{p}$ be a prime lying over $p$. Then, the inertia group $I_{\mathfrak{p}}$ is contained in the kernel of the map
\begin{equation*}
\rho_{E,\ell} : \Gal(\overline{\mathbb{Q}}/\mathbb{Q}) \to \operatorname{Aut}(E[\ell]) \cong \operatorname{GL}_2(\mathbb{Z}/\ell \mathbb{Z}),
\end{equation*}
where the kernel is $\ker \rho_{E,\ell} = \Gal(\overline{\mathbb{Q}}/\mathbb{Q}(E[\ell]))$. This implies that $\mathbb{Q}(E[\ell])$ is a subfield of the inertia field $\overline{\mathbb{Q}}^{I_{\mathfrak{p}}}$, and hence $\mathbb{Q}(E[\ell])$ is unramified at $p$.
\end{proof}

This corollary is also used as a consequence of the Néron–Ogg–Shafarevich criterion in the literature.

\begin{lemma}\label{43isogeny}
Let $\ell \in \{11, 19, 43, 67, 163\}$ be a prime, and let $p$ be a prime. Then, there exists an elliptic curve $E/\mathbb{Q}$ satisfying
\begin{equation*}
E(\mathbb{Q}(\zeta_p))_{\text{tors}} \cong \mathbb{Z}/\ell \mathbb{Z}
\end{equation*}
if and only if $p = \ell$.
\end{lemma}

\begin{proof}
Assume that $E(\mathbb{Q}(\zeta_p))_{\text{tors}} \cong \mathbb{Z}/\ell \mathbb{Z}$ for some prime $p$ and some $\ell \in \{11, 19, 37, 43, 67, 163\}$.
 By Lemma \ref{n-isogeny}, $E$ must have an $\ell$-isogeny. In Chapter 7 of \cite{Chou}, the $j$-invariants of such elliptic curves are listed. In total, there are $7$ different $j$-invariants, and we can also identify the Cremona labels of the corresponding elliptic curves. Using the LMFDB database, we observe that the conductors of these elliptic curves are equal to $N_E = \ell^2$ when working with the minimal quadratic twists of the corresponding $j$-invariants. This fact, combined with Corollary \ref{unramifiedcorollary}, implies that $\Q(E[\ell])$ is unramified at all primes other than $\ell$. For the remainder of the proof, the elliptic curve $E/\Q$ will always refer to the minimal quadratic twist of its $j$-invariant. 

 Let $P = (x_0, y_0) \in E(\Q(\zeta_p))_{\text{tors}}$ be a point of order $\ell$. Since $\Q(P) \subset \Q(E[\ell])$, we can see that $\Q(P)$ is unramified outside $\ell$.
From Mazur's theorem, we know that the field $\Q(P)$ cannot be equal to $\Q$. Therefore, the prime $\ell$ is ramified in $\Q(P)$.
Due to our assumption, $\Q(E[\ell])$ is contained in the field $\Q(\zeta_p)$; thus, if $\ell$ is ramified in $\Q(\zeta_p)$, then $p = \ell$ must hold.
Indeed, \texttt{SageMath} calculations show that $E(\Q(\zeta_\ell))_{\text{tors}} \cong \Z/\ell\Z$ for the corresponding elliptic curves (except in the case $\ell = 163$, where allocated memory becomes a problem. Please see the following explanation for how to achieve the same result without computer calculations).

On the other hand, our elliptic curve has an $\ell$-isogeny; therefore, there exists a point $P = (x_0, y_0) \in E(\overline{\Q})_{\text{tors}}$ of order $\ell$, and the field $\Q(P)$ is a Galois extension whose Galois group $\Gal(\Q(P)/\Q)$ is isomorphic to a subgroup of $(\Z/\ell\Z)^\times$ by Lemma \ref{isogeny-fieldofdefinition}. Since the Galois group is finite abelian, the Kronecker–Weber theorem implies that there exists a positive integer $n$ such that $\Q(P) \subset \Q(\zeta_n)$. Let $n$ be the minimal such positive integer.

We know that $\Q(P)$ is unramified outside $\ell$, which implies that $n$ must be a power of $\ell$. Moreover, $[\Q(P):\Q]$ divides $\ell - 1$, so $n = \ell$ must hold. Together with Theorem \ref{allgroupslist}, this shows that $E(\Q(\zeta_\ell))_{\text{tors}} \cong \Z/\ell\Z$ for the corresponding elliptic curves.

We have so far proved the lemma for the minimal quadratic twist of the corresponding $j$-invariants. To extend the result to the quadratic twists of those elliptic curves, we consider the extension obtained by adjoining to $\Q$ the $x$-coordinates of the torsion points. Again, let $P = (x_0, y_0) \in E(\Q(\zeta_\ell))_{\text{tors}}$ be a point of order $\ell$, which we know exists by our calculations. The point $P$ is the generator of the kernel of an $\ell$-isogeny (from $E$ to some other elliptic curve).

We aim to show that if $E_d(\Q(\zeta_p))_{\text{tors}} \cong \Z/\ell\Z$ for some prime $p$ and square-free integer $d$, then $p = \ell$ must hold.

It is clear that $P_d = (x_0, y_0 / \sqrt{d}) \in E_d(\overline{\Q})_{\text{tors}}$, the image of $P$ on $E_d$, is the generator of an $\ell$-isogeny (from $E_d$ to some other elliptic curve). Moreover, by Lemma \ref{n-isogeny}, the torsion subgroup $E_d(\Q(\zeta_p))_{\text{tors}}$ is also the kernel of an $\ell$-isogeny. By Theorem \ref{Kenku}, an elliptic curve defined over $\Q$ cannot have two distinct rational $\ell$-isogenies for the primes $\ell \in \{7, 11, 13, 17, 19, 37, 43, 67, 163\}$. This implies that $E_d(\Q(\zeta_p))_{\text{tors}}$ must be generated by the point $P_d = (x_0, y_0 / \sqrt{d})$.
Then, the extensions 
\begin{equation*}
    \Q(x(P))=\Q(x(P_d))=\Q(x_0)
\end{equation*}
are the same. Again, if $\Q(x_0) \ne \Q$, then $\ell$ is ramified in $\Q(x_0)$, which is contained in $\Q(\zeta_p)$, implying $p = \ell$. If $x_0$ is rational, then there exists a quadratic twist $E_{d'}$ such that it has an $\ell$-torsion point over $\Q$, which contradicts Mazur’s theorem.

As a note, by working with the minimal quadratic twist or its $-\ell$-quadratic twist if necessary (see Table \ref{tablom}), and by consulting the LMFDB database, we observe that the field of definition of the $\ell$-torsion has extension degree $(\ell - 1)/2$, which implies that it is equal to $\Q(\zeta_\ell)^+$ for the elliptic curves we are considering. (Note: there is no data available for $\ell = 163$.) Table 2 of \cite{Chou} confirms the same result for all primes in our list.

Using the fact that
\begin{equation*}
  [\Q(P):\Q(x_0)] =  [\Q(x_0,y_0):\Q(x_0)]\leq 2,
\end{equation*}
and since $(\ell - 1)/2$ is odd, it follows that $\Q(x_0) = \Q(\zeta_\ell)^+$, where $x_0$ is the $x$-coordinate of the $\ell$-torsion point $P \in E(\Q(\zeta_\ell))_{\text{tors}}$. Indeed, $\ell$ is ramified in $\Q(x_0)$.
\end{proof}

\begin{lemma}\label{1737isogeny}
      Let $E/\mathbb{Q}$ be an elliptic curve, and let $\ell =17$ or $\ell=37$. Then there is no prime $p$ such that $E(\mathbb{Q}(\zeta_p))_{\text{tors}} \cong \mathbb{Z}/\ell \mathbb{Z}$.
\end{lemma}

\begin{proof}
Assume that $E(\Q(\zeta_p))_{\text{tors}} \cong \Z/\ell\Z$ for some prime $p$ and $\ell = 17$ or $\ell = 37$. By Lemma \ref{n-isogeny}, we know that $E$ has an $\ell$-isogeny. In Chapter 7 of \cite{Chou}, the $j$-invariants of such elliptic curves can be found. There are two different $j$-invariants for each case, totaling four. Let $E$ be the minimal quadratic twist of one of the corresponding $j$-invariants. We can also find their Cremona labels in the same chapter. Let $P = (x_0, y_0) \in E(\Q(\zeta_p))_{\text{tors}}$ be a point of order $\ell$.

In order to deal with all the quadratic twists at once and show that none of them can have $\ell$-torsion over $\Q(\zeta_p)$ for any prime $p$, we will show that the extension $\Q(x_0)$, obtained by adjoining $\Q$ with the $x$-coordinate of the $\ell$-torsion point, cannot be contained in any $\Q(\zeta_p)$.

If $x_0$ is rational, then $y_0^2$ is also rational (we are using the short Weierstrass form). Hence, $E_d(\Q)$ contains $\ell$-torsion for some quadratic twist $E_d$ of $E$, which contradicts Mazur’s theorem. Then, since $\Q(x_0) \subset \Q(\zeta_p)$, we can see that $p$ is ramified in $\Q(x_0)$. Using the inclusions
\begin{equation*}
\Q(x_0) \subset \Q(x_0, y_0) = \Q(P) \subset \Q(E[\ell]),  
\end{equation*}
we see that $p$ is ramified in $\Q(E[\ell])$. By Corollary \ref{unramifiedcorollary}, we observe that $p = \ell$ or $p \mid N_E$, where $N_E$ is the conductor of the elliptic curve.

This is where using the minimal quadratic twist of the corresponding $j$-invariant helps us. Even though the conductor changes under quadratic twists, we can work with one of them by focusing only on the $x$-coordinates.

In the $\ell = 17$ case, the minimal quadratic twist has conductor $N_E = 14450$ in both instances, so the only possibilities are $p = 5$ or $p = 17$. Similarly, in the $\ell = 37$ case, the minimal quadratic twist has conductor $N_E = 1225$ in both instances, so the only possibilities are $p = 5, 7, 37$.

To see that $\Q(x_0)$ is not contained in the field $\Q(\zeta_p)$ for the given primes, we factor the $\ell$-th division polynomial of the elliptic curves over $\Q(\zeta_p)$ using \texttt{SageMath} and observe that they do not have a root in those fields. This concludes our proof.
\end{proof}

\section{Relaxing the Conditions}\label{relaxingsection}
In this section, we will discuss the implications of relaxing the conditions on our prime $p$ when classifying $E(\mathbb{Q}(\zeta_p))_{\text{tors}}$. In our first theorem, we classified all possible torsion subgroups when $p-1$ is not divisible by $3$, $4$, or $5$. Now, we will gradually relax these conditions. As a result, the number of groups appearing in the classification will increase. One of the first theorems we obtain is the following.

\begin{theorem}\label{somerelaxation}
    Let $E/\mathbb{Q}$ be an elliptic curve.
Let $p > 3$ be a prime such that $p - 1$ is not divisible by $3$ or $4$. Then $E(\mathbb{Q}(\zeta_p))_{\text{tors}}$ is either one of the groups from Mazur’s Theorem, or one of the following groups:
\begin{equation*}
\mathbb{Z}/11\mathbb{Z}, \;\mathbb{Z}/16\mathbb{Z} , \;\mathbb{Z}/25\mathbb{Z}, \; \mathbb{Z}/2\mathbb{Z} \times \mathbb{Z}/10\mathbb{Z}, \; \text{or } \;\mathbb{Z}/2\mathbb{Z} \times \mathbb{Z}/12\mathbb{Z}.
\end{equation*}
 Among these, the groups $\mathbb{Z}/11\mathbb{Z}$, $\mathbb{Z}/25\mathbb{Z}$, $\mathbb{Z}/2\mathbb{Z} \times \mathbb{Z}/10\mathbb{Z}$, and $\mathbb{Z}/2\mathbb{Z} \times \mathbb{Z}/12\mathbb{Z}$ are realized for at least one prime $p$ satisfying the conditions given in the theorem. Moreover, the group $\mathbb{Z}/11\mathbb{Z}$ is realized only when $p = 11$. The group $\mathbb{Z}/16\mathbb{Z}$ is theoretically possible under these assumptions, but no example is currently known.
\end{theorem}

\begin{proof}
  The proof is the same as that of Theorem \ref{maintheorem}. The only difference is in the application of Lemma \ref{mainlemma} in the proof. Since we have removed the condition that $p - 1$ is not divisible by $5$, we cannot conclude that $\mathbb{Z}/11\mathbb{Z}$ or $\mathbb{Z}/25\mathbb{Z}$ are not realizable. We can use Lemma \ref{43isogeny} to see that $11$-torsion is realized over $\mathbb{Q}(\zeta_p)$ only when $p = 11$ and only for $3$ different $j$-invariants (see Table \ref{tablom}). The elliptic curve \lmfdbec{11}{a}{3} satisfies $E(\mathbb{Q}(\zeta_{11}))_{\text{tors}} \cong \mathbb{Z}/25\mathbb{Z}$. We could not find any other prime $p$ satisfying $E(\mathbb{Q}(\zeta_p))_{\text{tors}} \cong \mathbb{Z}/25\mathbb{Z}$ for some elliptic curve $E/\mathbb{Q}$, and it is beyond our scope to show that no other primes exist.
  The groups $\mathbb{Z}/2\mathbb{Z} \times \mathbb{Z}/10\mathbb{Z}$ and $\mathbb{Z}/2\mathbb{Z} \times \mathbb{Z}/12\mathbb{Z}$ are realized for $p=59$ and $p=47$ respectively (see Section \ref{realizationsection}).
\end{proof}

\begin{theorem}\label{mostmaintheorem}
     Let $E/\mathbb{Q}$ be an elliptic curve. 
Let $p>3$ be a prime such that $p-1$ is not divisible by $3$. Then $E(\mathbb{Q}(\zeta_p))_{\text{tors}}$ is either one of the groups from Mazur’s Theorem, or one of the following groups: 
$$E(\Q(\zeta_p))_{\text{tors}} \simeq 
  \begin{cases}
      \mathbb{Z} / N \mathbb{Z} & \text{ with } 1\leq N\leq 13, \text{ or } N=15,16,25, \\
           \mathbb{Z} / 2 \mathbb{Z} \times      \mathbb{Z} /2 N \mathbb{Z} & \text{ with } 1\leq N\leq 6, \text{ or } N=8,\\
           \mathbb{Z}/5\mathbb{Z} \times \mathbb{Z}/5\mathbb{Z}.\\
  \end{cases}  $$

  Among these, all groups are realized for at least one prime $p$ satisfying the conditions of the theorem, with the exception of the group $\mathbb{Z}/2\mathbb{Z} \times \mathbb{Z}/16\mathbb{Z}$, for which no example is currently known. 
  Moreover, the group $\Z/11 \Z$ is realized only when $p=11$. The groups $\Z/5\Z \times \Z/5\Z$ and $\Z/15\Z$ are realized only when $p=5$.
\end{theorem}

\begin{proof}
 The proof is the same as in Theorem \ref{maintheorem}. The only difference lies in the use of Lemma \ref{mainlemma} 
 during the proof. Since we removed the condition that $p - 1$ is not divisible by $4$ or $5$, we cannot rule out the 
 possibility that $\mathbb{Z}/N\mathbb{Z}$ is realizable for
 $N = 11, 13, 15,$ or $25$. Similarly, we cannot exclude the 
 possibility of $\mathbb{Z}/2\mathbb{Z} \times \mathbb{Z}/16\mathbb{Z}$. We also cannot rule out the cases
 $\mathbb{Z}/17\mathbb{Z}$ and $\mathbb{Z}/37\mathbb{Z}$ 
 using Lemma \ref{mainlemma}, but Lemma \ref{1737isogeny} allows us to exclude them. 
 
 One final exception is that we cannot eliminate the $\mathbb{Z}/5\mathbb{Z} \times \mathbb{Z}/5\mathbb{Z}$ case when $p = 5$, because the Weil pairing argument does not apply. Indeed, the elliptic curve \lmfdbec{11}{a}{2} satisfies $E(\mathbb{Q}(\zeta_{5}))_{\text{tors}} \cong \mathbb{Z}/5\mathbb{Z} \times \mathbb{Z}/5\mathbb{Z}$.

 By Lemma \ref{43isogeny}, $11$-torsion is realized over $\mathbb{Q}(\zeta_p)$ only when $p = 11$. 
The elliptic curve \lmfdbec{11}{a}{3} satisfies $E(\mathbb{Q}(\zeta_{11}))_{\text{tors}} \cong \mathbb{Z}/25\mathbb{Z}$. 

The elliptic curve \lmfdbec{50}{b}{3} satisfies $E(\mathbb{Q}(\zeta_5))_{\text{tors}} \cong \mathbb{Z}/15\mathbb{Z}$. In the proof of our next theorem, we show that $E(\mathbb{Q}(\zeta_p))_{\text{tors}} \cong \mathbb{Z}/15\mathbb{Z}$ is realized only when $p = 5$.

The elliptic curve \lmfdbec{2890}{a}{2} satisfies $E(\mathbb{Q}(\zeta_{17}))_{\text{tors}} \cong \mathbb{Z}/13\mathbb{Z}$. This data is available on the 
LMFDB. The field of definition of the $13$-torsion point is a
quartic field with discriminant $17^3$ and Galois group 
isomorphic to $\mathbb{Z}/4\mathbb{Z}$. Using the ramified 
primes and the Kronecker–Weber theorem, we can deduce that
this field is a subfield of $\mathbb{Q}(\zeta_{17})$. By 
Theorem~\ref{allgroupslist}, we conclude that the torsion 
subgroup over $\mathbb{Q}(\zeta_{17})$ must be isomorphic to 
$\mathbb{Z}/13\mathbb{Z}$, since a larger torsion subgroup is not possible.

We have no examples of $\mathbb{Z}/2\mathbb{Z} \times \mathbb{Z}/16\mathbb{Z}$ for any prime $p$, and our methods are not effective in ruling out its existence when $p\equiv 1\pmod{4}$. 

 The groups $\mathbb{Z}/2\mathbb{Z} \times \mathbb{Z}/10\mathbb{Z}$, $\mathbb{Z}/2\mathbb{Z} \times \mathbb{Z}/12\mathbb{Z}$ and $\Z/16\Z$ are realized for $p=59$, $p=47$ and $p=41$ respectively (see Section \ref{realizationsection}).
\end{proof}

\begin{theorem}\label{everycase}
     Let $E/\mathbb{Q}$ be an elliptic curve. 
Let $p>3$ be a prime. Then $E(\mathbb{Q}(\zeta_p))_{\text{tors}}$ is either one of the groups from Mazur’s Theorem, or one of the following groups: 
$$E(\Q(\zeta_p))_{\text{tors}} \simeq 
  \begin{cases}
      \mathbb{Z} / N \mathbb{Z} & \text{ with } 1\leq N\leq 16, \text{ or } N=18,19,25,43,67,163, \\
           \mathbb{Z} / 2 \mathbb{Z} \times      \mathbb{Z} /2 N \mathbb{Z} & \text{ with } 1\leq N\leq 9,\\
               \mathbb{Z} / 5 \mathbb{Z} \times      \mathbb{Z} /5 \mathbb{Z}.\\
  \end{cases}  $$
    Among these, all groups are realized for at least one prime $p>3$, with the exception of the group $\mathbb{Z}/2\mathbb{Z} \times \mathbb{Z}/16\mathbb{Z}$, for which no example is currently known. 
  Moreover, for $\ell\in\{11,19,43,67,163\}$, the group $\Z/\ell \Z$ is realized only when $p=\ell$. The groups $\Z/5\Z \times \Z/5\Z$ and $\Z/15\Z$ are realized only when $p=5$. The group $\Z/14\Z$ is realized only when $p=7$.
\end{theorem}

\begin{proof}
 Although the proof shares similarities with that of Theorem \ref{maintheorem}, Lemma \ref{mainlemma} is not applicable here. Instead, we eliminate groups from the list in Theorem \ref{allgroupslist} by analyzing each group type individually.

The groups of the form $\mathbb{Z}/N\mathbb{Z} \times \mathbb{Z}/MN\mathbb{Z}$ are not realized when $N \in \{3,4,5,6,8\}$ because $\zeta_N$ is not contained in $\mathbb{Q}(\zeta_p)$ for $p > 5$. On the other hand, the torsion subgroup $\mathbb{Z}/5\mathbb{Z} \times \mathbb{Z}/5\mathbb{Z}$ is realized over $\mathbb{Q}(\zeta_5)$, with the elliptic curve \lmfdbec{11}{a}{2} serving as an example.

Among the remaining subgroups, $\mathbb{Z}/\ell \mathbb{Z}$ is realized only when $p = \ell$, for $\ell \in \{11, 19, 43, 67, 163\}$, as shown by Lemma \ref{43isogeny}. Moreover, Lemma \ref{1737isogeny} establishes that $\mathbb{Z}/17\mathbb{Z}$ and $\mathbb{Z}/37\mathbb{Z}$ are not realized over any $\mathbb{Q}(\zeta_p)$ for any prime $p$.

Now, we show that $\mathbb{Z}/N\mathbb{Z}$ is not realized for any prime $p$ when $N \in \{21, 27\}$, while $N=14$ is realized only for $p=7$, and $N=15$ only for $p=5$. Assume that $E(\mathbb{Q}(\zeta_p))_{\text{tors}} \cong \mathbb{Z}/N\mathbb{Z}$ for some prime $p$. Then, by Lemma \ref{n-isogeny}, $E$ must admit an $N$-isogeny.

By examining Table 2 in \cite{Chou} and Table 4 in \cite{fieldofdefinition}, we see that this occurs for only eleven different $j$-invariants. Using the LMFDB database, we analyze the field of definition of their $N$-torsion, where our elliptic curves correspond to the minimal quadratic twists of these $j$-invariants, or their $-3$-quadratic twists in some cases where this is more convenient (see Table \ref{tablom}).

In the case of $N=14$, there are two distinct $j$-invariants. When working with their minimal quadratic twists, the torsion growth observed on the LMFDB database shows that the field of definition for the $14$-torsion with an abelian Galois group is exactly $\mathbb{Q}(\zeta_7)$. Let us focus on the unique $7$-isogeny. We see that its kernel is defined over $\Q(\zeta_7)$. This means there exists a point $P=(x_0,y_0)\in E(\overline{\Q})_{\text{tors}}$ of order $7$, with $\Q(P)=\Q(\zeta_7)$. Moreover, the set
\begin{equation*}
\langle P\rangle= \{aP\in E(\overline{\Q})_{\text{tors}}:1\leq a\leq 7\} \subset E(\overline{\Q})_{\text{tors}},
\end{equation*}
is the kernel of the unique (uniqueness follows from Theorem \ref{Kenku}) $7$-isogeny (from $E$ to another elliptic curve), by Lemma \ref{n-isogeny}. Considering the fact
\begin{equation*}
[\Q(P):\Q(x_0)] = [\Q(x_0,y_0):\Q(x_0)] \leq 2,
\end{equation*}
we see that $\Q(x_0)$ has extension degree $3$ or $6$ over $\Q$. Therefore, $\Q(\zeta_7)^+ \subset \Q(x_0)$ and $7$ is ramified at $\Q(x_0)$.

Let $E_d(\Q(\zeta_p))_{\text{tors}}\cong \Z/14\Z$ for some prime $p$ and square-free integer $d$. Then, $E_d$ has a $7$-isogeny by Lemma \ref{n-isogeny}, and it is unique by Theorem \ref{Kenku}. Thus, its kernel is generated by the point $P_d=(x_0,y_0/\sqrt{d})\in E_d(\overline{\Q})_{\text{tors}}$, which is the image of $P$ on the quadratic twist $E_d$. Then, $P_d\in E_d(\Q(\zeta_p))_{\text{tors}}$ implies $\Q(x_0)\subset \Q(\zeta_p)$, implying that $7$ is ramified at $\Q(\zeta_p)$. This implies that $14$-torsion is not realized over $\mathbb{Q}(\zeta_p)$ when $p \neq 7$.

For the case of $N = 15$, there are $4$ different $j$-invariants with a $15$-isogeny. Let $E/\Q$ have one of the corresponding $j$-invariants; then there is a point $P=(x_0,y_0)\in E(\overline{\Q})_{\text{tors}}$ of order $15$ such that the set
\begin{equation*}
\langle P\rangle= \{aP\in E(\overline{\Q})_{\text{tors}}:1\leq a\leq 15\} \subset E(\overline{\Q})_{\text{tors}}
\end{equation*}
is the kernel of a $15$-isogeny (from $E$ to some 
other elliptic curve). Moreover, this $15$-isogeny 
is unique due to Theorem \ref{Kenku}. This means 
that $\langle P_d \rangle \subset E_d(\overline{\Q})_{\text{tors}}$ is the kernel of the unique $15$-isogeny of $E_d$. Here, as usual, $P_d=(x_0,y_0/\sqrt{d})$ denotes the image of $P$ on the
quadratic twist $E_d$. Then, to deal with all 
quadratic twists at once, we can pick one of them 
and consider the extension $\Q(x_0)$, because it is 
the same extension for all quadratic twists of $E$. 
Our \texttt{SageMath} calculations show that $5$ is ramified 
at $\Q(x_0)$ for all $4$ different $j$-invariants.
Thus, $5$ must be ramified at $\Q(\zeta_p)$, 
considering $\Q(x_0)\subset \Q(\zeta_p)$. This 
implies that $15$-torsion is not realized over 
$\mathbb{Q}(\zeta_p)$ when $p \neq 5$.

For the case of $N = 21$, there are $4$ different $j$-invariants with a $21$-isogeny. Let $E/\Q$ have one of the corresponding $j$-invariants; then there is a point $P=(x_0,y_0)\in E(\overline{\Q})_{\text{tors}}$ of order $21$ such that the set
\begin{equation*}
\langle P\rangle= \{aP\in E(\overline{\Q})_{\text{tors}}:1\leq a\leq 21\} \subset E(\overline{\Q})_{\text{tors}}
\end{equation*}
is the kernel of a $21$-isogeny (from $E$ to some other elliptic curve). Moreover, this $21$-isogeny is unique due to Theorem \ref{Kenku}. This means that $\langle P_d \rangle \subset E_d(\overline{\Q})_{\text{tors}}$ is the kernel of the unique $21$-isogeny of $E_d$. Here, as usual, $P_d=(x_0,y_0/\sqrt{d})$ denotes the image of $P$ on the quadratic twist $E_d$. Then, to deal with all quadratic twists at once, we can pick one of them and consider the extension $\Q(x_0)$, because it is the same extension for all quadratic twists of $E$. Our \texttt{SageMath} calculations show that $3$ is ramified at $\Q(x_0)$ for all $4$ different $j$-invariants. Thus, $3$ must be ramified at $\Q(\zeta_p)$, considering $\Q(x_0)\subset \Q(\zeta_p)$. This implies that $21$-torsion is not realized over $\mathbb{Q}(\zeta_p)$ when $p \neq 3$. By Theorem \ref{najmanquadratic}, we see that $E(\Q(\zeta_3))_{\text{tors}}\cong \Z/21\Z$ is not possible.

Finally, for the case of $N = 27$, we note that there is only one rational $j$-invariant with a $27$-isogeny, namely when $j = -2^{15} \cdot 3 \cdot 5^3$. This corresponds to the elliptic curve \lmfdbec{27}{a}{2} and its quadratic twists. Moreover, for any elliptic curve $E/\Q$ with $j(E)= -2^{15} \cdot 3 \cdot 5^3$, the $27$-isogeny is unique due to Theorem \ref{Kenku}. Let $E$ denote the minimal quadratic twist and $P=(x_0,y_0)\in E(\overline{\Q})_{\text{tors}}$ be a point of order $27$ generating the kernel of the $27$-isogeny.

By using the LMFDB we can see that $\Q(P)=\mathbb{Q}(\zeta_{27})^+$. Considering the fact
\begin{equation*}
[\Q(P):\Q(x_0)]= [\Q(x_0,y_0):\Q(x_0)]\leq 2,
\end{equation*}
and $[\Q(\zeta_{27})^+:\Q]=9$, we can see that $\Q(x_0)=\Q(\zeta_{27})^+$. In any quadratic twist $E_d$, the unique $27$-isogeny is generated by the point $P_d=(x,y/\sqrt{d})\in E_d(\overline{\Q})_{\text{tors}}$, hence the extension $\Q(x_0)$ remains the same. Therefore, $27$-torsion is not realized over $\mathbb{Q}(\zeta_p)$ for any prime $p$.

 The groups $\mathbb{Z}/2\mathbb{Z} \times \mathbb{Z}/10\mathbb{Z}$, $\mathbb{Z}/2\mathbb{Z} \times \mathbb{Z}/12\mathbb{Z}$ and $\Z/16\Z$ are realized for $p=59$, $p=47$ and $p=41$ respectively (see Section \ref{realizationsection}).
\end{proof}

We discuss, in more detail, the realization of torsion group structures other than those appearing in Mazur’s theorem in the next section.

With the theorem provided, we can classify the torsion subgroups of elliptic curves defined over $\mathbb{Q}(\zeta_p)$ for primes $p > 3$. To complete the classification, we can refer to the work of Gu\v{z}vi\'{c} and Vukorepa in \cite{guzvicvukorepa}, who classified the torsion subgroups for $p = 5, 7, 11$. By incorporating their results, we can enhance the classification for these primes.

While our result offers a general classification for any $p > 3$, the methods we used also provide a way to refine the list of realizable torsion subgroups for specific primes. By focusing on a particular prime $p$, one can further improve the classification and identify the torsion subgroups realized over $\mathbb{Q}(\zeta_p)$. Although it is possible to obtain an exact list for each prime $p$, the task becomes increasingly difficult as $p$ grows larger.

\begin{theorem}[Gu\v{z}vi\'{c}, Vukorepa, \cite{guzvicvukorepa}]
    Let $E/\Q$ be an elliptic curve and let $p \in \{5,7,11\}$ be a prime number. Apart from the
groups in Mazur’s theorem, the group $E(\Q(\zeta_p))_{\text{tors}}$ can only be isomorphic to one of the following
groups:
\begin{enumerate}
    \item If $p=5$, 
    \begin{equation*}
      \Z/15\Z,  \quad \Z/16\Z, \text{ and} \quad    \Z/5\Z \times \Z/5\Z.
    \end{equation*}

     \item If $p=7$, 
    \begin{equation*}
        \Z/13\Z, \quad \Z/14\Z, \quad \Z/18\Z, \quad
        \Z/2\Z \times \Z/14\Z,  \text{ and} \quad \Z/2\Z \times \Z/18\Z.
    \end{equation*}

     \item If $p=11$, 
    \begin{equation*}
       \Z/11\Z,  \quad \Z/25\Z, \text{ and} \quad    \Z/2\Z \times \Z/10\Z.
    \end{equation*}

\end{enumerate}
\end{theorem}

For the case of $p = 3$, one can refer to Najman's results in
\cite{NajmanQuadratic}, where he provides a general 
classification of torsion subgroups over all quadratic 
fields. In a separate work \cite{najmanquadraticcyclotomic}, 
he gives a more specific classification for the field 
$\mathbb{Q}(\sqrt{-3})$ (and also for $\mathbb{Q}(\sqrt{-1})$), which directly addresses the case $p = 3$.
This classification, along with the general methods 
discussed, helps refine the understanding of torsion 
subgroups for this particular prime.

For $p = 2$, we have the trivial case where $\mathbb{Q}(\zeta_2) = \mathbb{Q}$, and thus the classification of torsion subgroups for elliptic curves over $\mathbb{Q}$ is given by Mazur's theorem. Mazur's theorem provides the possible torsion subgroups over the rationals, which is a well-established and complete classification.

In summary, the cases $p = 2$ and $p = 3$ are well-covered by Mazur's theorem and Najman's work, respectively, while the cases $p = 5, 7, 11$ are covered by Gu\v{z}vi\'{c} and Vukorepa. Our approach for $p > 3$ provides a more general classification, along with methods to further refine the possibilities for specific primes.

We conclude this section by presenting a comprehensive table that summarizes Theorems \ref{maintheorem}, \ref{somerelaxation}, \ref{mostmaintheorem}, and \ref{everycase}. This table includes the conditions on the primes, the exceptional torsion subgroups (i.e., those not appearing in Mazur’s theorem) that can occur, and specifies for which groups explicit examples are known or remain unknown. While some examples have been provided within the proofs of these theorems, a more detailed discussion of these examples, along with additional constructions, will be presented in the next section.

\begin{table}[H]
\centering
\small 
\caption{Summary of exceptional torsion subgroups over cyclotomic fields and their realizability}
\begin{tabular}{|>{\raggedright\arraybackslash}p{4cm}|
                >{\raggedright\arraybackslash}p{8.8cm}|
                >{\raggedright\arraybackslash}p{3cm}|}
\hline
\textbf{Divisibility conditions on primes $p>3$} & \textbf{Possible exceptional groups} & \textbf{Groups without known examples} \\
\hline
$p-1$ is not divisible by $3$, $4$, or $5$ & 
$\mathbb{Z}/16\mathbb{Z}$, 
$\mathbb{Z}/2\mathbb{Z} \times \mathbb{Z}/10\mathbb{Z}$, 
$\mathbb{Z}/2\mathbb{Z} \times \mathbb{Z}/12\mathbb{Z}$ & 
$\mathbb{Z}/16\mathbb{Z}$ \\
\hline
$p-1$ is not divisible by $3$ or $4$ & 
\begin{tabular}{@{}l@{}}
$\mathbb{Z}/N\mathbb{Z}$ with $N=11,16,25$, \\
$\mathbb{Z}/2\mathbb{Z} \times \mathbb{Z}/2N\mathbb{Z}$ with $N=5,6$
\end{tabular} & 
$\mathbb{Z}/16\mathbb{Z}$ \\
\hline
$p-1$ is not divisible by $3$ & 
\begin{tabular}{@{}l@{}}
$\mathbb{Z}/N\mathbb{Z}$ with $N=11,13,15,16,25$, \\
$\mathbb{Z}/2\mathbb{Z} \times \mathbb{Z}/2N\mathbb{Z}$ with $N=5,6,8$, \\
$\mathbb{Z}/5\mathbb{Z} \times \mathbb{Z}/5\mathbb{Z}$
\end{tabular} & 
$\mathbb{Z}/2\mathbb{Z} \times \mathbb{Z}/16\mathbb{Z}$ \\
\hline
Any prime & 
\begin{tabular}{@{}l@{}}
$\mathbb{Z}/N\mathbb{Z}$ with $N=11,13,14,15,16,18,19,25,43,67,163$, \\
$\mathbb{Z}/2\mathbb{Z} \times \mathbb{Z}/2N\mathbb{Z}$ with $N=5,6,7,8,9$, \\
$\mathbb{Z}/5\mathbb{Z} \times \mathbb{Z}/5\mathbb{Z}$
\end{tabular} & 
$\mathbb{Z}/2\mathbb{Z} \times \mathbb{Z}/16\mathbb{Z}$ \\
\hline
\end{tabular}
\label{tab:torsion-summary}
\end{table}

\section{Realization of Torsion Subgroups}\label{realizationsection}

In this section, we begin by discussing the realization of the three torsion subgroups that appear in Theorem \ref{maintheorem} but are not included in Mazur's theorem. These subgroups must be examined in the context of quadratic subfields. The first result in this section demonstrates that these torsion subgroups can indeed be realized over quadratic fields. We then present two subsequent theorems that provide methods for identifying and locating these torsion subgroups within such fields.

Later in the section, we expand the classification by considering torsion subgroups that arise when the conditions of Theorem \ref{maintheorem} are relaxed. In particular, Lemma \ref{43isogeny} has already shown that certain torsion subgroups are realized over $\mathbb{Q}(\zeta_p)$ for a unique prime $p$. We further explore how relaxing specific conditions permits the realization of additional torsion subgroups over particular number fields.

\begin{theorem}
   Let $E/\mathbb{Q}$ be an elliptic curve. 
Let $p>3$ be a prime such that $p-1$ is not divisible by $3,4,5,7,11$. If $E(\mathbb{Q}(\zeta_p))_{\text{tors}}$ is isomorphic to one of the following groups $\mathbb{Z}/2\mathbb{Z} \times \mathbb{Z}/10\mathbb{Z}, \mathbb{Z}/2\mathbb{Z} \times \mathbb{Z}/12\mathbb{Z}$ or $\mathbb{Z}/16\mathbb{Z}$, then $E(\mathbb{Q}(\zeta_p))_{\text{tors}} = E(\mathbb{Q}(\sqrt{-p}))_{\text{tors}}$. 
\end{theorem}
\begin{proof}
   Firstly, let 
   \begin{equation*}
       E(\mathbb{Q}(\zeta_p))_{\text{tors}} \cong \Z/2\Z \times \Z/2N\Z \quad \text{for } N = 5, 6.
   \end{equation*}
   The only quadratic subfield of $\Q(\zeta_p)$ is $\Q(\sqrt{-p})$. Combining this fact with Lemma \ref{4pairinggeneralmore}, we obtain 
\begin{equation*}
    E(\Q(\sqrt{-p}))[N] = E(\Q(\zeta_p))[N].
\end{equation*}
   Assume that $E$ is in short Weierstrass form. Then $\Q(E[2])$ is the splitting field of the cubic $x^3 + Ax + B$. Since $E(\Q(\zeta_p))_{\text{tors}}$ has full $2$-torsion, the field $\Q(E[2])$ must be a subfield of $\Q(\zeta_p)$. Then the cubic cannot be irreducible, since $p - 1$ is not divisible by $3$. This means that either the cubic has three rational roots, or it has one rational root and two other roots from $\Q(\sqrt{-p})$. This implies that $E(\Q(\sqrt{-p}))_{\text{tors}}$ also has full $2$-torsion. Ultimately, $E(\Q(\zeta_p))_{\text{tors}} = E(\Q(\sqrt{-p}))_{\text{tors}}$.

Secondly, let $E(\mathbb{Q}(\zeta_p))_{\text{tors}} \cong \Z/16\Z$. Let $P = (x_0, y_0) \in E(\mathbb{Q}(\zeta_p))_{\text{tors}}$ be a point of order $16$. Then, the extension degree of $\Q(x_0)$ is equal to the number of elements in the set $\{x_0^\sigma : \sigma \in \Gal(\Q(\zeta_p)/\Q)\}$, and it also divides $n$. In total, there are exactly $8$ points of order $16$ in $E(\mathbb{Q}(\zeta_p))_{\text{tors}}$, and their $x$-coordinates form a set of $4$ elements. This means that $\{x_0^\sigma : \sigma \in \Gal(\Q(\zeta_p)/\Q)\}$ has at most $4$ elements. It cannot have $3$ or $4$ elements because we assumed $n$ is not divisible by them. This implies that $\Q(x_0)$ is at most a degree $2$ extension, so $x_0 \in \Q(\sqrt{-p})$. Also, $[\Q(x_0, y_0) : \Q(x_0)] \leq 2$, so the extension degree of $\Q(x_0, y_0)$ is a divisor of $4$, but it also divides $n$, so it is at most a degree $2$ extension. This implies $P \in E(\Q(\sqrt{-p}))_{\text{tors}}$. Since $P$ generates $E(\mathbb{Q}(\zeta_p))_{\text{tors}}$, we obtain $E(\mathbb{Q}(\zeta_p))_{\text{tors}} = E(\mathbb{Q}(\sqrt{-p}))_{\text{tors}}$.
\end{proof}

This theorem shows that if $E(\Q(\zeta_p))_{\text{tors}}$ is not one of the groups in Mazur's theorem, then we should consider only the quadratic extension and look for torsion points there. These three groups appear in Theorem \ref{najmanquadratic}, so we know that they are realizable over infinitely many elliptic curves. However, we are interested in whether they are realizable over the quadratic fields $\Q(\sqrt{-p})$ for primes $p$ such that $p - 1$ is not divisible by $3$, $4$, $5$, $7$, or $11$. 

By using the next two theorems, we can show that they are indeed realizable over such quadratic fields for some prime $p$.

\begin{theorem}[Jeon, Kim, Lee, \cite{JeonKimLee}, Theorem 3.2]
    Put $d=d(t)=8t^3-8t^2+1$ with $t \in \mathbb{Q}.$ Let $E$ be an elliptic curve
defined by the equation
\begin{equation*}E:y^2 + (1 - c)xy - by = x^3 - bx^2\end{equation*}
where $b=\frac{t^3(2t^2-3t+1)}{(t^2-3t+1)^2}$ and $c= -\frac{t(2t^2-3t+1)}{t^2-3t+1}$ with $t\neq 0,\frac{1}{2},1.$ Then the torsion subgroup of $E$ over $\mathbb{Q}(\sqrt{d})$ is $\mathbb{Z}/2\mathbb{Z} \times \mathbb{Z}/10\mathbb{Z}.$
\end{theorem}

In this theorem, if we choose $t = -75/242$, then we get $d = -59 \times 14^2 / 11^6$, which means that $E$ has torsion $\mathbb{Z}/2\mathbb{Z} \times \mathbb{Z}/10\mathbb{Z}$ over $\mathbb{Q}(\sqrt{-59})$. This torsion subgroup is contained in $E(\mathbb{Q}(\zeta_{59}))_{\text{tors}}$. By Theorem \ref{maintheorem}, we observe that $E(\mathbb{Q}(\zeta_{59}))_{\text{tors}} \cong \mathbb{Z}/2\mathbb{Z} \times \mathbb{Z}/10\mathbb{Z}$, since it cannot be a larger group.

\begin{theorem}[Jeon, Kim, Lee, \cite{JeonKimLee}, Theorem 3.3]
Put $d = d(t) = \frac{t^2-1}{t^2+3}$ with $t\in \mathbb{Q}$. Let $E$ be an elliptic curve defined
by the equation    
\begin{equation*}E:y^2 + (1 - c)xy - (c + c^2)y = x^3 - (c + c^2)x^2\end{equation*}
where $c=\frac{1-t^2}{t^4+3t^2}$ with $t\neq -1,0,1.$ Then the torsion subgroup of $E$ over $\mathbb{Q}(\sqrt{d})$ is equal to $\mathbb{Z}/2\mathbb{Z} \times \mathbb{Z}/12\mathbb{Z}.$
\end{theorem}

In this theorem, if we choose $t = \frac{248354}{307104}$, then we get $d = -47\left(\frac{5^2 \times 17 \times 31}{37 \times 7933}\right)^2$. This means that the curve described above has torsion $\mathbb{Z}/2\mathbb{Z} \times \mathbb{Z}/12\mathbb{Z}$ over $\mathbb{Q}(\sqrt{-47})$. This torsion subgroup is contained in $E(\mathbb{Q}(\zeta_{47}))_{\text{tors}}$. By Theorem \ref{maintheorem}, we observe that $E(\mathbb{Q}(\zeta_{47}))_{\text{tors}} \cong \mathbb{Z}/2\mathbb{Z} \times \mathbb{Z}/12\mathbb{Z}$, since it cannot be a larger group.

These two elliptic curves show that the group structures
$\mathbb{Z}/2\mathbb{Z} \times \mathbb{Z}/10\mathbb{Z}$ and
$\mathbb{Z}/2\mathbb{Z} \times \mathbb{Z}/12\mathbb{Z}$ are
realizable. However, note that we have shown they are realizable for
specific values of $p$. We do not know if they are realizable for all
such primes $p$.

Torsion of elliptic curves over $\mathbb{Q}(\sqrt{d})$ is studied by Trbovi\'{c} in \cite{trbovic} for $0 < d < 100$. If a similar study is conducted in the future for $\mathbb{Q}(\sqrt{-d})$ instead, this would lead to a classification of torsion over $\mathbb{Q}(\zeta_p)$ for small primes like $p=47$ or $p=59$, where $p-1$ is not divisible by $3,4,5,7,11$.

The realization of $\mathbb{Z}/16\mathbb{Z}$ as a torsion subgroup over $\mathbb{Q}(\zeta_p)$ remains an open problem in general. In order for this group to appear as a torsion subgroup, we need to find a prime $p$ such that
$E(\mathbb{Q}(\sqrt{-p}))_{\text{tors}} \cong \mathbb{Z}/16\mathbb{Z}$ when $p \equiv 3 \pmod{4}$.
According to Theorem 1.3 in Banwait and Derickx \cite{bander}, such a realization does not occur for any prime $p < 800$ with $p \equiv 3 \pmod{4}$. This rules out all the primes smaller than $800$ considered in Theorem~\ref{maintheorem}.

Their result further shows that among the primes $p \equiv 1 \pmod{4}$, the only prime less than $800$ for which this torsion structure is realized is $p = 41$. For example, the elliptic curve \lmfdbec{266910}{ck}{6} satisfies
$E(\mathbb{Q}(\sqrt{41}))_{\text{tors}} \cong \mathbb{Z}/16\mathbb{Z}$ (available on LMFDB). This fact, combined with Theorem~\ref{allgroupslist}, implies that the torsion subgroup
$E(\mathbb{Q}(\zeta_{41}))_{\text{tors}}$ is isomorphic to $\mathbb{Z}/16\mathbb{Z}$.

Now, we turn our attention to the subgroups that arise when we relax the conditions in Theorem \ref{maintheorem} and are realized over only finitely many elliptic curves. These groups are of the form $\mathbb{Z}/N\mathbb{Z}$, where the modular curve $X_0(N)$ has genus greater than zero and possesses non-cuspidal rational points. For example, we are not interested in the case $N=20$, even though $X_0(20)$ has genus one and six rational points, because all of these points are cusps. Therefore, there is no rational elliptic curve with a $20$-isogeny.

We will be considering the values $N\in\{11,14,15,19,21,27,37,43,67,163\}$, and these values correspond to finitely many elliptic curves with $N$-isogeny. 
Here, when we say finitely many elliptic curves, we actually mean finitely many distinct $j$-invariants. Indeed, we must be careful because when working with a quadratic twist of an elliptic curve $E/\Q$, we can change the field of definition of the 
$N$-torsion points and even the extension degree can change. More rigorously, we have the isomorphism 
$$E:y^2=x^3+Ax+B \rightarrow E_d = dy^2=x^3+Ax+B,$$ 
defined by 
$$P=(x,y)\in E\rightarrow P_d=(x,y/\sqrt{d})\in E_d,$$
where $A, B$ are rationals and $d$ is a square-free integer. This implies that the field of definition $\mathbb{Q}(P_d)$ is contained in $\mathbb{Q}(P, \sqrt{d})$. As an example, we have the elliptic curve \lmfdbec{121}{a}{2} and its $11$-torsion is defined over $\mathbb{Q}(\zeta_{11})$. If we consider its $-11$-quadratic twist \lmfdbec{121}{c}{1}, then its $11$-torsion is defined over $\mathbb{Q}(\zeta_{11})^+$. If we consider its $-3$-quadratic twist \lmfdbec{1089}{i}{2}, then its $11$-torsion is defined over $\mathbb{Q}(\zeta_{33})^+$.
This data is available on LMFDB.

Our goal is to classify torsion subgroups over $ \mathbb{Q}(\zeta_p) $ where $ p $ is a prime. In the table below, we will list the elliptic curves $ E/\mathbb{Q} $ with $ E(K)_{\text{tors}} \cong \mathbb{Z}/N\mathbb{Z} $, for the values of $ N $ mentioned earlier, where $K$ is an abelian number field. We will also share the field of definition of the $N$-torsion. By Lemma \ref{isogeny-fieldofdefinition}, we know that the field of definition of the kernel of an $N$-isogeny is an abelian number field, so it is contained in some $\mathbb{Q}(\zeta_n)$ due to the Kronecker–Weber theorem. This is why we have an elliptic curve for each $j$-invariant with an $N$-isogeny in the table. We are also presenting the minimum positive integer $n$ described above.

As discussed earlier, the field of definition of a point of order $N$ can vary depending on the choice of the point. In our analysis, we choose a point such that the field of definition is an abelian number field. Another important consideration is the choice of quadratic twist. Depending on the square-free integer $d$, it is possible to have $E(\mathbb{Q}(\zeta_p))_{\text{tors}} \cong \mathbb{Z}/N\mathbb{Z}$ while $E_d(\mathbb{Q}(\zeta_p))_{\text{tors}} \not\cong \mathbb{Z}/N\mathbb{Z}$, since the field of definition of the $N$-torsion may be contained in $\mathbb{Q}(\zeta_p, \sqrt{d})$.

If possible, we will choose a suitable quadratic twist such that $E(\mathbb{Q}(\zeta_p))_{\text{tors}} \cong \mathbb{Z}/N\mathbb{Z}$ for some prime $p$. If this is not possible, we will list the field of definition of the $N$-torsion for the minimal quadratic twist corresponding to the given $j$-invariant. If multiple quadratic twists satisfy the same condition, we will select the one for which the field of definition has the smallest extension degree. For each case, we indicate the specific quadratic twist used relative to the minimal twist. We also include the conductor of the chosen twist, as it is relevant to the torsion data.

We will refer to the $j$-invariants under consideration, which can be found in Table~2 of Section~7 in \cite{Chou} and Table~4 of Section~9 in \cite{fieldofdefinition}. Throughout, we use LMFDB labels to identify elliptic curves, and the relevant data can be accessed on the LMFDB database \cite{lmfdb}.

Some explanation is needed regarding how the data in the table are obtained. In fact, almost no computations are involved. For the cases $N = 11, 14, 15, 17, 19, 21, 27, 43$, we simply present the information available in the growth of torsion section for each elliptic curve on the LMFDB. When we list an exact field of definition, it is stated explicitly there.

An important observation can be made in the case $N = 14$, which is the only even integer among the values discussed above. If we choose the elliptic curve $E$ as \lmfdbec{49}{a}{4}, then we obtain 
\begin{equation*}
E(\Q(\zeta_7))_{\text{tors}}\cong \Z/2\Z\times \Z/14\Z.\end{equation*}
However, if we instead consider the elliptic curve $E$ as \lmfdbec{49}{a}{3}, then we find
\begin{equation*}
E(\mathbb{Q}(\zeta_7))_{\text{tors}} \cong \mathbb{Z}/14\mathbb{Z}.
\end{equation*}
This shows that the two exceptional group structures appearing in Theorem~\ref{everycase} are indeed realized by some rational elliptic curves.

In the cases $N = 67$ and $N = 163$, the LMFDB database does not specify the exact field of definition of the $N$-torsion. However, for $N = 67$, it does state that the extension degree of the field is equal to $(N-1)/2$. Combining this fact with the discussion in the proof of Lemma~\ref{43isogeny}, we can conclude that the field of definition is $\mathbb{Q}(\zeta_N)^+$, as it is the unique subfield of $\mathbb{Q}(\zeta_N)$ of degree $(N-1)/2$ over $\mathbb{Q}$. For $N = 163$, all we can say is that the field of definition is contained in $\mathbb{Q}(\zeta_N)$.

For $N = 17$ and $N = 21$, the field of definition of the $N$-torsion is contained in $\mathbb{Q}(\zeta_{85})$ and $\mathbb{Q}(\zeta_{21})$, respectively. On the LMFDB database, the Galois groups and discriminants of these fields are provided. By combining these data with the Kronecker–Weber theorem, we obtain the information presented in the table.

The last case is $N = 37$. If $j(E) = -7 \cdot 11^3$, then the field of definition of the $N$-torsion is $\mathbb{Q}(\zeta_{35})^+$ when $E$ is the minimal quadratic twist. This data is available on the LMFDB. 

If $j(E) = -7 \cdot 137^3 \cdot 2083^3$, the situation becomes more intriguing. We know that the field of definition of the $37$-torsion is an abelian number field contained in $\mathbb{Q}(\zeta_n)$ for some positive integer $n$. To determine this minimal value of $n$, we apply Corollary~\ref{unramifiedcorollary}, using the conductor $N_E=5^2\cdot7^2$ of the minimal quadratic twist available on the LMFDB. The database also indicates that the extension degree of the field of definition is $36$ over $\mathbb{Q}$. From this, we deduce that $n$ must divide $1295$.

To confirm that $n$ cannot be smaller, we use \texttt{SageMath} to factor the $37$-division polynomial of $E$ over $\mathbb{Q}$. This polynomial factors into four irreducible components over $\mathbb{Q}$: one of degree $18$ and three of degree $222$. Let $P = (x_0, y_0) \in E(\overline{\mathbb{Q}})_{\text{tors}}$ be the generator of the kernel of the $37$-isogeny from $E$ to another elliptic curve. Then the degree $[\mathbb{Q}(x_0): \mathbb{Q}]$ must divide $36$. Hence, $\mathbb{Q}(x_0)$ has degree $18$, and its minimal polynomial corresponds to the first factor of the $37$-division polynomial.

We then examine the splitting field of this polynomial and observe that $5$, $7$, and $37$ are ramified in the extension. This confirms that the minimal $n$ for which the field of definition is contained in $\mathbb{Q}(\zeta_n)$ is indeed $1295$ for the minimal quadratic twist corresponding to this $j$-invariant.

\renewcommand{\arraystretch}{1.1}
\begin{table}[ht]
\centering
\caption{Elliptic curves with $N$-torsion defined over an abelian number field}\label{tablom}
\scalebox{0.8}{
\begin{tabular}{|c|c|c|c|c|c|}\hline

$N$& $j$-invariant & Quadratic Twist & LMFDB Label & Conductor & Field of Definition of the $N$-torsion \\ \hline

\multirow{ 3}{*}{11} & $-11 \cdot 131^3$ & -11& \lmfdbec{121}{c}{1}& $11^2$ & $\Q(\zeta_{11})^+$ \\ \cline{2-6}
&$-2^{15}$ & Minimal& \lmfdbec{121}{b}{2}& $11^2$  & $\Q(\zeta_{11})^+$  \\ \cline{2-6}
&$-11^2$ & -11& \lmfdbec{121}{a}{1}& $11^2$ & $\Q(\zeta_{11})^+$ \\ \hline

\multirow{ 2}{*}{14} & $-3^3 \cdot 5^3$ & $-7$& \lmfdbec{49}{a}{2}& $7^2$& $\Q(\zeta_{7})^+$\\ \cline{2-6}
&$3^3\cdot 5^3 \cdot 17^3$ & $-7$ & \lmfdbec{49}{a}{1}& $7^2$ & $\Q(\zeta_{7})^+$\\
 \hline

 \multirow{ 4}{*}{15} & $-5^2 / 2$ & Minimal& \lmfdbec{50}{a}{3}& $2\cdot 5^2$ & $\Q(\zeta_{5})$\\ \cline{2-6}
 & $-5^2 \cdot 241^3 / 2^3$ & -3& \lmfdbec{450}{g}{1}& $2\cdot 3^2 \cdot 5^2$& $\Q(\zeta_{15})^+$\\ \cline{2-6}
 & $-5 \cdot 29^3 / 2^5$ & Minimal& \lmfdbec{50}{b}{3}& $2\cdot 5^2$& $\Q(\sqrt{5})$\\ \cline{2-6}
&$5 \cdot 211^3 / 2^{15}$ &Minimal& \lmfdbec{50}{b}{4}& $2\cdot 5^2$& $\Q(\sqrt{-15})$\\
 \hline

 \multirow{ 2}{*}{17} & $-17^2 \cdot 101^3 / 2$ & Minimal& \lmfdbec{14450}{b}{2}& $2\cdot 5^2 \cdot 17^2$& Contained in $\Q(\zeta_{85})$\\ \cline{2-6}
& $-17 \cdot 373^3 / 2^{17}$ &Minimal& \lmfdbec{14450}{o}{2}& $2\cdot 5^2 \cdot 17^2$ & Contained in $\Q(\zeta_{85})$\\
 \hline

19& $-2^{15} \cdot 3^3$ & Minimal & \lmfdbec{361}{a}{2} & $19^2$& $\Q(\zeta_{19})^+
$ \\ \hline

\multirow{ 4}{*}{21} & $-3^2 \cdot 5^6 / 2^3$ &-3 & \lmfdbec{162}{c}{3}& $2\cdot 3^4$& $\Q(\zeta_{9})^+$\\ \cline{2-6}
 & $3^3 \cdot 5^3 / 2$ & Minimal& \lmfdbec{162}{b}{4}& $2\cdot 3^4$& $\Q(\zeta_{9})$\\ \cline{2-6}
 & $-3^2 \cdot 5^3 \cdot 101^3/2^{21}$ & -3& \lmfdbec{162}{c}{2}  &   $2\cdot 3^4$& Contained in $\Q(\zeta_{21})$\\ \cline{2-6}
&$-3^3 \cdot 5^3 \cdot 383^3 / 2^7$ &Minimal& \lmfdbec{162}{b}{1}&  $2\cdot 3^4$ & Contained in $\Q(\zeta_{21})$\\
 \hline

27& $-2^{15} \cdot 3 \cdot 5^3$ & Minimal & \lmfdbec{27}{a}{2} & $3^3$& $\Q(\zeta_{27})^+
$ \\ \hline

\multirow{ 2}{*}{37} & $-7 \cdot 11^3$ & Minimal& \lmfdbec{1225}{b}{2}& $5^2\cdot 7^2$ & $\Q(\zeta_{35})^+$\\ \cline{2-6}
&  $-7 \cdot 137^3 \cdot 2083^3$ &Minimal& \lmfdbec{1225}{b}{1}& $5^2\cdot 7^2$ & Contained in $\Q(\zeta_{1295})$\\
 \hline

43& $-2^{18} \cdot 3^3 \cdot 5^3$ & Minimal & \lmfdbec{1849}{b}{2} & $43^2$& $\Q(\zeta_{43})^+
$ \\ \hline

67& $-2^{15} \cdot 3^3 \cdot 5^3 \cdot 11^3$ & Minimal & \lmfdbec{4489}{b}{2} & $67^2$& $\Q(\zeta_{67})^+
$ \\ \hline

163& $-2^{18} \cdot 3^3 \cdot 5^3 \cdot 23^3 \cdot 29^3$ & Minimal & \lmfdbec{26569}{a}{2} & $163^2$& Contained in $\Q(\zeta_{163})$\\ \hline

\end{tabular}
}

\end{table}

Finally, we shall discuss the groups $\mathbb{Z}/N\mathbb{Z}$ where $N \in \{13, 16, 18, 25\}$. These groups are special among the groups of the form $\mathbb{Z}/N\mathbb{Z}$, because the remaining such groups either appear in the statement of Mazur's theorem or are realized for only finitely many elliptic curves. In contrast, the four group structures under consideration are realized for infinitely many elliptic curves, since the modular curve $X_0(N)$ has genus zero for these values of $N$.

We have already discussed $\mathbb{Z}/16\mathbb{Z}$ earlier in this chapter and provided an example where $16$-torsion is realized over $\mathbb{Q}(\zeta_{41})$. For the other cases, the elliptic curve \lmfdbec{147}{c}{2} satisfies $E(\mathbb{Q}(\zeta_7)^+)_{\text{tors}} \cong \mathbb{Z}/13\mathbb{Z}$, the elliptic curve \lmfdbec{14}{a}{4} satisfies $E(\mathbb{Q}(\zeta_7)^+)_{\text{tors}} \cong \mathbb{Z}/18\mathbb{Z}$, and the elliptic curve \lmfdbec{11}{a}{3} satisfies $E(\mathbb{Q}(\zeta_{11})^+)_{\text{tors}} \cong \mathbb{Z}/25\mathbb{Z}$.

The problem of determining the primes $p$ for which there exists an elliptic curve $E/\mathbb{Q}$ such that $E(\mathbb{Q}(\zeta_p))_{\text{tors}} \cong \mathbb{Z}/N\mathbb{Z}$ for $N \in \{13, 16, 18, 25\}$ is a difficult one, as we discussed earlier even in the case $N = 16$ alone, and lies beyond the scope of this work.

The case $N=18$ is particularly interesting and warrants a discussion similar to that of the 
$N=14$ case, since both group structures 
$\Z/N\Z$ and 
$\Z/2\Z \times \Z/2N\Z$ can occur. For the first structure, we have the elliptic curve 
$E$ given as \lmfdbec{14}{a}{4}, which satisfies
\begin{equation*}
E(\mathbb{Q}(\zeta_7))_{\text{tors}} = E(\mathbb{Q}(\zeta_7)^+)_{\text{tors}} \cong \mathbb{Z}/18\mathbb{Z}.
\end{equation*}
For the second structure, the elliptic curve 
$E$ given as \lmfdbec{14}{a}{5} satisfies
\begin{equation*}
E(\mathbb{Q}(\zeta_7))_{\text{tors}} \cong \mathbb{Z}/2\mathbb{Z} \times \mathbb{Z}/18\mathbb{Z}.
\end{equation*}

The last example among the exceptional groups listed in Theorem~\ref{everycase} of the form $ \mathbb{Z}/2\mathbb{Z} \times \mathbb{Z}/2N\mathbb{Z} $ occurs when $N=8$. We were unable to find any rational elliptic curve $E$ satisfying $ E(\mathbb{Q}(\zeta_p))_{\text{tors}} \cong \mathbb{Z}/2\mathbb{Z} \times \mathbb{Z}/16\mathbb{Z} $ for any prime $p$. In Lemma
\ref{mod16}, we showed that such a realization is not possible if $p\equiv 3\pmod{4}$.
However, our methods are not strong enough to rule out the existence of such a curve when $p \equiv 1 \pmod{4}$, and we believe it might still be possible in this case. In \cite{Chouquartic}, it is shown that there exist abelian quartic number fields $K$ for which $E(K)_{\text{tors}} \cong \mathbb{Z}/2\mathbb{Z} \times \mathbb{Z}/16\mathbb{Z}$. For instance, the elliptic curve \lmfdbec{15}{a}{6} satisfies this condition for the field $K = \mathbb{Q}(\zeta_{15})^+$.

\end{document}